\documentclass[12pt,reqno]{article}

\usepackage{hyperref}

\usepackage{amssymb}
\usepackage{amsmath}
\usepackage{amsfonts}
\usepackage{latexsym}
\usepackage{mathrsfs}
\usepackage{graphicx}
\usepackage{indentfirst,xspace,color}
\usepackage{pgf,tikz}
\usepackage{pstricks-add}
\usepackage{pgfplots}

\newcommand{\R}{{\mathbb R}}

\newtheorem{theorem}{Theorem}[section]
\newtheorem{lemma}{Lemma}[section]

\newenvironment{proof}{\noindent{\bf Proof:}}{\hfill$\Box$}

\title{A bifurcation problem for a one-dimensional $p-$Laplace elliptic problem with non-odd absorption}
\author{A. N. Carvalho\thanks{%
		With the support of  by NSFC Grant 11671367} \\
	\textit{andcarva@icmc.usp.br}
	\and
	T. L. M. Luna\thanks{%
		With the support of  by FAPESP Grant 2019/20341-3.}           \\
	\textit{titoluna@icmc.usp.br}
	}
\begin{document}

\maketitle
{\footnotesize

\centerline{Instituto de Ci\^{e}ncias Ma\-te\-m\'{a}\-ti\-cas e de Computa\c{c}\~{a}o, }

\centerline{Universidade de S\~{a}o Paulo-Campus de S\~{a}o Carlos}

\centerline{Caixa Postal 668, S\~{a}o Carlos SP, Brazil}
}

\bigskip

\centerline{(Communicated by the associate editor name)}

\begin{abstract}
	
	In this paper we study the existence of solutions of a  one-dimensional eigenvalue problem 	 
	$-\left(|\phi_x|^{p-2}\phi_x\right)_x=\lambda \left(|\phi|^{q-2}\phi-f(\phi)\right)$ such that $\phi(0)=\phi(1)=0$, where $p,q>1$, $\lambda$ is a positive real parameter and $f$ is a continuous (not necessarily odd) function.  Our goal is to give a complete description of solutions of this problem. We completely characterize the set of solutions of this problem, which may be uncountable. For $1<p\neq 2$, the existing results treat only the case when $f$ is either odd and a power (see \cite{TAYA}) or when $p=q$ (\cite{Guedda-Veron}). Our method of proof rely on a careful analysis of the phase diagram associated with this equation, refining the regularity results of \cite{otani} and characterizing the exact points where we may have $C^2$ regularity of solutions including some points $\chi\in (0,1)$ for which $\phi_x(\chi)=0$.
	
 \end{abstract}
	
\section{Introduction}

Determining the structure of the set of stationary solutions is a very important step for the understanding of the asymptotic behavior of the solutions of the associated to the autonomous parabolic problem
\begin{equation}\label{parabolic}
\begin{cases}
u_t-\left(|u_x|^{p-2}u_x\right)_x=\lambda\left( |u|^{q-2}u-f(u)\right), &\text{in $(0,1)$},\\
u(0,t)=u(1,t)=0
\\ 
u(x,0)=u_0(x), \ u_0\in W^{1,p}(0,1). 
\end{cases}	
\end{equation}
In fact, the semigroup associated to \eqref{parabolic} is gradient. This implies that any solution of \eqref{parabolic} must converge to a connected component of the set of stationary solutions and that any global solution, that is, defined in the whole real line, must converge forwards and backwards to connected components of the set of stationary solutions.

\bigskip
	
In this paper we consider the following nonlinear eigenvalue problem
\begin{equation}\label{equilibria}
\begin{cases}
-\left(|\phi_x|^{p-2}\phi_x\right)_x=\lambda\left( |\phi|^{q-2}\phi-f(\phi)\right), &\text{in $(0,1)$},\\ 
\phi(0)=\phi(1)=0 &
\end{cases}	
\end{equation}
where, $p,q>1$, $\lambda>0$ is a real parameter and $f\in C(\R)$  is a (not necessarily odd) function such that
the function $|s|^{q-2}s-f(s)$ is zero at zero, has a first positive zero $z^+$ and a first negative zero $z^-$ and such that the function  $g(s)=\frac{f(s)}{|s|^{q-2}s}$,  $s\in(z^-,z^+)$,  satisfies
\begin{equation}\label{hypothesis-monotononicty}
	\text{$g$ is strictly decreasing in $(z^-,0)$ and strictly increasing in $(0,z^+)$}
\end{equation}
and
\begin{equation}\label{hypothesis-zero}
		\displaystyle\lim_{s\to 0} g(s)=0, \qquad \lim_{s\to (z^\pm)^\mp}\left(\frac{|s|^{q-2}s-f(s)-|z^\pm|^{q-2}z^\pm+f(z^\pm)}{|s|^{q-2}s-|z^\pm|^{q-2}z^\pm}\right)=L<0.
\end{equation}

\bigskip

Our aim is to study the structure of set $E_\lambda$ of the solutions of \eqref{equilibria}, this structure depends on the parameter  $\lambda$ and changes as $\lambda$ grows through of certain sequences of positive real numbers 
$\lambda_n$ and $\tilde{\lambda}_n$, that is, as $\lambda$ increases through $\lambda_n (\tilde{\lambda}_n)$, a bifurcation occur giving rise to a pair (a pair of connected components with a continuum) of solutions of \eqref{equilibria} that vanishes exactly $n+1$ times in $[0,1]$.

\bigskip

Let us mention in some more detail the results that have motivated our investigation. The case $f(s)=|s|^{q+r-2}s$ with $r>0$ (odd and a power function), $q\geqslant2$ and $p>2$, has been discussed in the work of Takeuchi and Yamada in \cite{TAYA}. They have obtained the sequences 
\begin{equation*}
	\begin{split}
		\lambda_n=
		\begin{cases}
			0,&\text{if $q<p$}\\
			\displaystyle n^p(p-1)\left[2\int^1_0(1-t^p)^{-1/p}dt\right]^p,&\text{if $q=p$}\\
			\displaystyle\frac{p-1}{p}n^p\left[2\int^{a^*}_0\left(\frac{t^{q+r}-(a^*)^{q+r}}{q+r}+\frac{(a^*)^q-t^{q}}{q}\right)^{-\frac{1}{p}}dt\right]^p,&\text{if $q> p$}
		\end{cases},& \quad n=1,2,\cdots
	\end{split}
\end{equation*}
and
\begin{equation*}
	\begin{split}
		\tilde{\lambda}_n=\frac{p-1}{p}n^p\left[2\int^1_0\left(\frac{t^{q+r}-1}{q+r}+\frac{1-t^{q}}{q}\right)^{-\frac{1}{p}}dt\right]^p,&\quad\text{$n=1,2,\cdots$}
	\end{split}
\end{equation*}
where $a^*\in (0,1)$ is the unique point where the function 
$$(0,1)\ni a\mapsto \int^a_0\left(\frac{t^{q+r}-a^{q+r}}{q+r}+\frac{a^q-t^{q}}{q}\right)^{-\frac{1}{p}}dt
$$ 
attains its minimum.

\begin{theorem}\label{theorem-ty} Let $q<p$. The structure of $E_\lambda$  of solutions of \eqref{equilibria} is given by
	$$E_\lambda=\{0\}\cup E^\pm_1\cup E^\pm_2\cup\cdots$$
	where 
	\begin{equation}\label{new-form2}
		E_j^\pm=
		\begin{cases}
			\{\phi^\pm_j\},&\text{if $0<\lambda\leqslant\tilde{\lambda}_n$}\\
			[0,1]^{j-1},&\text{if $\tilde{\lambda}_n<\lambda$.}
		\end{cases}
	\end{equation}
	for $j=1,2,\cdots$
\end{theorem}

\begin{theorem}\label{theorem-ty} Let $q>p$. The structure of $E_\lambda$  of solutions of \eqref{equilibria} is given by
	\begin{enumerate}
		\item[(i)] If $\lambda<\lambda_1,$ then $E_\lambda=\{0\}$, and
		\item[(ii)] If $\lambda_n\leqslant\lambda<\lambda_{n+1},$ for some $n=1,2,\cdots$ then
		$$E_\lambda=\{0\}\cup E^\pm_1\cup\cdots\cup E^\pm_n$$
		where 
		\begin{equation}\label{new-form3}
			E_j^\pm=
			\begin{cases}
				\{\phi^\pm_j,\psi^\pm_j\},&\text{if $\lambda_n<\lambda\leqslant\tilde{\lambda}_n$}\\
				[0,1]^{2j-2},&\text{if $\tilde{\lambda}_n<\lambda$.}
			\end{cases}
		\end{equation}
		for $j=1,\cdots,n$, and $E_n^\pm=\{\phi^\pm_n\}$ since $\lambda_n=\lambda.$
	\end{enumerate}
\end{theorem}

\bigskip

 In \cite{Guedda-Veron}, Guedda and Veron studied the following elliptic problem
\begin{equation}\label{equilibriagueddaveron}
	\begin{cases}
		-\left(|\phi_x|^{p-2}\phi_x\right)_x=\lambda |\phi|^{p-2}\phi-f(\phi), &\text{in $(0,1)$},\\ 
		\phi(0)=\phi(1)=0 &
	\end{cases}	
\end{equation} where $p>1$, $\lambda$ is a real number and $f$ is a $C^1$ real-valued odd function such that $s\mapsto \frac{f(s)}{|s|^{p-2}s}$ is increasing in $(0,+\infty)$ with limits $0$ at $0$ and $+\infty$ at $+\infty$. They describe the structure of the bifurcations of the solutions of \eqref{equilibriagueddaveron} as $\lambda$ increases through the eigenvalues of the one-dimensional p-Laplace operator on $W^{1,p}_0(0,1)$, that is, the sequence $\lambda_n=\displaystyle n^p(p-1)\left[2\int^1_0(1-t^p)^{-1/p}dt\right]^p$ for $n=1,2,\cdots$. More explicitly, when $1 < p \leqslant 2$ there is still a finite number of solutions of \eqref{equilibriagueddaveron} such as in the semilinear case $p = 2$ and, when $p>2$, they show that the number of corresponding solutions of the problem \eqref{equilibriagueddaveron} may vary from one to a continuum set of solutions. In fact they show the  following
\begin{theorem}\label{theorem-gv} The structure of $E_\lambda$ is given by
	\begin{enumerate}
		\item[(i)] If $\lambda\leqslant\lambda_1,$ $E_\lambda=\{0\}$, and
		\item[(ii)] If $1<p\leqslant2$ and $\lambda_n<\lambda\leqslant\lambda_{n+1},$ for some $n=1,2,\cdots$ then
		$$E_\lambda=\{0\}\cup\{\phi^\pm_1,\cdots,\phi^\pm_n\}$$
		where 
		\begin{equation*}
			\phi^\pm_j\in Z^\pm_j\left.=
			\begin{cases}
				\varphi\in S^\pm_j:&\text{$\varphi(x)=\varphi\left(\frac{1}{j}-x\right)$ for $x\in(0,2j)$ and }\\
				&\text{$\varphi(x)=-\varphi\left(x-\frac{1}{j}\right)$ for $x\in(\frac{1}{j},1)$}\\
			\end{cases}\right\}
		\end{equation*}
		with
		\begin{equation}\label{set-zeros}
			S^\pm_j\left.=
			\begin{cases}
				\varphi\in C^1([0,1]):&\text{$\varphi(0)=\varphi(1)=0,$ $\pm\varphi_x(0)>0$ and}\\
				&\text{$\varphi$ has exactly $j-1$ simple zeros in $(0,1)$}\\
			\end{cases}\right\}
		\end{equation}
		for $j=1,\cdots,n.$
		\item[(iii)] If $2<p$ and $\lambda_n<\lambda\leqslant\lambda_{n+1},$ for some $n=1,2,\cdots$ then
		$$E_\lambda=\{0\}\cup\{\phi^\pm_1\}\displaystyle\cup E^\pm_2\cup\cdots\cup E^\pm_n$$ 
		where $\phi^\pm_1\in Z^\pm_1$ and each $E^\pm_j\subset S^\pm_j$ for $j=2,\cdots,n$ satisfies the following properties 
		\begin{itemize}
			\item $E^\pm_j$ is a set of a single element  if $2jx(\lambda)\geqslant1$ and
			\item $E^\pm_j$ is diffeomorphic to $[0,1]^{j-1}$ if $0<2jx(\lambda)<1$
		\end{itemize}
		where  
		\begin{equation}
			x(\lambda)=\int^{h(\lambda)}_0\left(\frac{\lambda}{p-1}h^p(\lambda)-\frac{p}{(p-1)}F(h(\lambda))+\frac{p}{p-1}F(s)-\frac{\lambda}{p-1}s^p\right)^{-1/p}ds
		\end{equation}since  $h(\lambda)$ is such that $f(h(\lambda))-\lambda(h(\lambda))^{p-1}=0$  and again $F(s)=\int^s_0f(\xi)d\xi$.
	\end{enumerate} 
\end{theorem}

\bigskip

We observe that \eqref{equilibriagueddaveron} always has a unique solution in $Z^\pm_j$, it is  $\phi^\pm_j\in Z^\pm_j$ for $j=1,\cdots,n$, since $f$ is odd-function. The number $x(\lambda)$ in the Theorem  \ref{theorem-gv}, that goes to $0$ as $\lambda$ becomes large,  gives origin to a second sequence corresponding to a secondary bifurcation to a continuous of solutions \eqref{equilibriagueddaveron}. When $f$ is the power nonlinearity, $s\mapsto b|s|^{r-2}s$, with $r>p$ and $b>0$, we have that this second sequence is $\tilde{\lambda}_n=\displaystyle n^p(p-1)\left[2\int^1_0(1-t^p-\frac{p}{r}(1-t^r))^{-1/p}dt\right]^p>\lambda_n$  for $n=1,2,\cdots$. Thus we can rewrite the statements $iii)$ to the Theorem \ref{theorem-gv} as, 
\begin{equation}\label{new-form}
	E_j^\pm=
	\begin{cases}
		\{\phi^\pm_j\},&\text{if $\lambda_n<\lambda\leqslant\tilde{\lambda}_n$}\\
		[0,1]^{j-1},&\text{if $\tilde{\lambda}_n<\lambda$.}
	\end{cases}
\end{equation}

\bigskip

Complementing the results of  \cite{TAYA,Guedda-Veron}, that is, when $f$ is either odd and a power (see \cite{TAYA}) or when $p=q$ (\cite{Guedda-Veron}) we consider more general functions $f$ which are not necessarily a power, not necessarily odd and $p$ is not necessarily equal to $q$. 
We show that we may also obtain the sequence of bifurcations that we will present next. To that end we define $r(\lambda)$ and $r^-(\lambda)$ are the positive numbers given by
\begin{equation}\label{max-inclination-positive}
r(\lambda)^p=\frac{\lambda p}{(p-1)} \left(\frac{|z^+|^q}{q} -F(z^+)\right).
\end{equation}
\begin{equation}\label{max-inclination-negative}
r^-(\lambda)^p=\frac{\lambda p}{(p-1)} \left(\frac{|z^-|^q}{q} -F(z^-)\right).
\end{equation}
$r^*(\lambda)=\min\{r(\lambda),r^-(\lambda)\}$ and $z(\cdot),S(\cdot),I(\cdot)$ and $J(\cdot)$ are the functions defined in \eqref{zero}, \eqref{negativezero}, \eqref{I} and \eqref{J}, respectively.

\bigskip

\underline{\textbf{If $1<q<p$};} (see Figure \ref{bifphe-qlessp}) the bifurcation phenomena is determined by the sequences

\begin{equation*}
	\lambda_n=0,\qquad n=1,2,\cdots
\end{equation*}
and $\tilde{\lambda}_n$ has the following structure
 \begin{equation}\label{lambda-tilde-N}
	\begin{cases}
		\lambda^\pm_n=+\infty,&\text{$n=1,2,\cdots$,  $1<p\leqslant2$}\\
		\lambda^+_{1}=\displaystyle\frac{p-1}{p}\left[2I(z^+))\right]^p,&\text{$2<p$  }\\
		\lambda^-_{1}=\displaystyle\frac{p-1}{p}\left[2J(z^-))\right]^p,&\text{$2<p$ }\\
		\lambda^\pm_{2k}=\displaystyle\frac{p-1}{p}\left[2kI(z(r^*(\lambda)))+2kJ(S(r^*(\lambda)))\right]^p,&\text{$k=1,2,\cdots$, $2<p$}\\
		\lambda^+_{2k-1}=	\displaystyle	\frac{p-1}{p}\left[2(k-1)I(z(r^*(\lambda)))+2kJ(S(r^*(\lambda)))\right]^p,&\text{$k=2,3,\cdots$, $2<p$}\\
		\lambda^-_{2k-1}=	\displaystyle	\frac{p-1}{p}\left[2kI(z(r^*(\lambda)))+2(k-1)J(S(r^*(\lambda)))\right]^p,&\text{$k=2,3,\cdots$, $2<p$.}
	\end{cases}
\end{equation}
We note that $\lambda^\pm_{n}<\lambda^\pm_{n+1}$ for $n=1,2,\cdots$.

\bigskip

\underline{\textbf{If $1<p<q$};} (see Figures \ref{bifphe-plessq-1}, \ref{bifphe-plessq-2}, \ref{bifphe-plessq-3} or Theorem \ref{Theorem-qgreatthanp}), the bifurcation phenomena is completely determined by the sequences  $\tilde{\lambda}_n$ (see \eqref{lambda-tilde-N}) and $\lambda_n$ is given by
\begin{equation}\label{lambda-star-N}
	\begin{cases}
		\lambda^+_{*,1}=\displaystyle\frac{p-1}{p}\left[2I(a_*))\right]^p,&\\
		\lambda^-_{*,1}=\displaystyle\frac{p-1}{p}\left[2J(b_*))\right]^p,&\\
		\lambda^\pm_{*,2k}=\displaystyle\frac{p-1}{p}\left[2kI_e\right]^p,&\text{$k=1,2,\cdots$,}\\
		\lambda^+_{*,2k-1}=	\displaystyle	\frac{p-1}{p}\left[2kI^+_o+2(k-1)J^+_o\right]^p,&\text{$k=2,3,\cdots$,}\\
		\lambda^-_{*,2k-1}=	\displaystyle	\frac{p-1}{p}\left[2(k-1)J^-_o+2kI^-_o\right]^p,&\text{$k=2,3,\cdots$,}
	\end{cases}
\end{equation}
where $a_*$ and $b_*$ are the points where the functions $I(\cdot)$ and $J(\cdot)$ attain their minimum, respectively, $I_e$ satisfies \eqref{Ie} and $I^+_o,J^+_o$ satisfy \eqref{IJpositiveodd} and $\lambda^\pm_{*,2k}<\lambda^\pm_{*,2(k+1)}$ for $k=1,2,\cdots$,  $\lambda^\pm_{*,2k-1}<\lambda^\pm_{*,2(k+1)-1}$ $k=2,3,\cdots$ and $\lambda^\pm_{*,n}<\lambda^\pm_{n}$ for each $n=1,2,\cdots$.

\bigskip

\underline{\textbf{If $1<q=p$};} (e.g. see  Figure \ref{bifphe-qequalp} or Theorem \ref{Theorem-qequalp}), the bifurcation phenomena is determined by the sequences $\lambda_n$ of eigenvalues of the one dimensional $p$-Laplace operator with Dirichlet boundary condition and $\tilde{\lambda}_n>\lambda_n$ defined by \eqref{lambda-tilde-N}. 

\bigskip

Thus, there are  the following theorems 
\begin{theorem}\label{Theorem-qequalp}
	Under the above assumptions on $f$ and $1<q=p<+\infty $ 
	\begin{enumerate}
		\item[]{$(i)$} If $\lambda \leqslant \lambda_1$ then the structure of solutions of \eqref{equilibria} is $E_\lambda=\{0\}$
		\item[]{$(ii)$} If $\lambda_n < \lambda \leqslant \lambda_{n+1}$,  for some $n=1,2,\cdots$, then
		
			$$
			E_\lambda=\{0\}\cup\bigcup^n_{k=1}E^\pm_k
			$$
		
		where $E^\pm_k\subset S^\pm_k$ satisfies 
		\begin{equation}\label{continuous-bifurcation}
			E^\pm_k=
			\begin{cases}
				\{\phi^\pm_k\},& \text{if $\lambda_k<\lambda\leqslant\lambda^\pm_k$,}\\
				\displaystyle[0,1]^{\mathcal{E}^\pm(k)},& \text{if $\lambda^\pm_k<\lambda$.}
			\end{cases}
		\end{equation}
		for $k=1,\cdots,n$, where $\mathcal{E}^\pm$ is defined in
		
		\begin{table}[htbp]
			\centering
			\caption{Definition of $\mathcal{E}^\pm$ - Area Condition}\label{tablaEk}
			\bigskip
			\begin{tabular}{|c|c|c|c|c|}		 \hline
				Definition & $A(z^+)= A(z^-)$ & $A(z^+)< A(z^-)$&$A(z^+)> A(z^-)$& for\\ \hline
				$[0,1]^0:=$ & $\{0\}$ & $\{0\}$&$\{0\}$ &$k=1$\\   \hline
				$\mathcal{E}^\pm(2k)=$ & $2k-1$ & $k-1$&$k-1$ &$k=1,2,\cdots$\\   \hline
				$\mathcal{E}^+(2k-1)=$ & $2k-2$ & $k-1$ &$k-2$& $k=2,3,\cdots$\\   \hline
				$\mathcal{E}^-(2k-1)=$ & $2k-2$ & $k-2$ &$k-1$ & $k=2,3,\cdots$\\   \hline
			\end{tabular}
		\end{table}
	where $A(z^+)=\frac{(z^+)^q}{q}-F(z^+)$ $($respectively $A(z^-)=\frac{|z^-|^q}{q}-F(z^-))$ is the area between of the function $|s|^{q-2}s-f(s)$ and the $s$ axis on the set $(0,z^+)$ $($respectively $(z^-,0))$.
	\end{enumerate}
\end{theorem}

\begin{theorem}\label{Theorem-qlessp}
	Under the above assumptions on $f$ and $1<q<p<+\infty$ for each $\lambda>0$
		$$
		E_\lambda=\{0\}\cup\bigcup^\infty_{k=1}E^\pm_k
		$$
		where $E^\pm_k\subset S^\pm_k$ satisfies \eqref{continuous-bifurcation}.
\end{theorem}
and
\begin{theorem}\label{Theorem-qgreatthanp}
	Under the above assumptions on $f$ and $1<p<q<+\infty $ 
	\begin{enumerate}
		\item[]{$(i)$} If $\lambda < \min\{\lambda^+_{*,1},\lambda^-_{*,1}\}$ then the structure of solutions of \eqref{equilibria} is $E_\lambda=\{0\}$
		\item[]{$(ii)$} If $\lambda^\pm_{*,2n} \leqslant \lambda < \lambda^\pm_{*,2(n+1)}$,  for some $n=1,2,\cdots$, then
		
			$$
			E_\lambda=\{0\}\cup\bigcup^n_{k=1}E^\pm_k\cup\bigcup^{n}_{k=1}\{\psi^\pm_k\}
			$$
		
		where $\psi_k\in S^\pm_k$ is a solution of \eqref{equilibria} for each $k=1, \cdots,n$ and  $E^\pm_k\subset S^\pm_k$ satisfies
		\begin{equation}\label{bifurcationform3}
			E^\pm_k=
			\begin{cases}
				\{\psi^\pm_k,\}& \text{if $\lambda^\pm_{*,k}=\lambda$,}\\
				\{\phi^\pm_k\},& \text{if $\lambda^\pm_{*,k}<\lambda\leqslant\lambda^\pm_k$,}\\
				\displaystyle[0,1]^{\mathcal{E}^\pm(k)},& \text{if $\lambda^\pm_k<\lambda$.}
			\end{cases}
		\end{equation}
		for $k=1, \cdots,n$.
		
		\item[]{$(iii)$} If $\lambda^\pm_{*,2n-1} \leqslant \lambda < \lambda^\pm_{*,2(n+1)-1}$,  for some $n=1,2,\cdots$, then
		
			$$
			E_\lambda=\{0\}\cup\bigcup^n_{k=1}E^\pm_k\cup\bigcup^{n}_{k=1}\{\psi^\pm_k\}
			$$
		
		where $\psi_k\in S^\pm_k$ is a solution of \eqref{equilibria} for each $k=1, \cdots,n$ and  $E^\pm_k$ satisfies \eqref{bifurcationform3}.

	\end{enumerate}
\end{theorem}

\begin{figure}
	\centering
	
	\begin{tikzpicture}[scale=0.6,rotate=-90]
		
		\clip(-6.26,-0.7) rectangle (6.4,14.21);
		
		\draw[color=blue] [samples=50,rotate around={0:(0,0)},xshift=0cm,yshift=0cm,line width=6pt,domain=-5.789473684210526:-2.52)] plot (\x,{(\x)^2/2/0.2631578947368421});
		\draw [samples=50,rotate around={0:(0,0)},xshift=0cm,yshift=0cm,line width=2pt,domain=-2.52:2.52)] plot (\x,{(\x)^2/2/0.2631578947368421});
		\draw[color=blue] [samples=50,rotate around={0:(0,0)},xshift=0cm,yshift=0cm,line width=6pt,domain=2.52:5.789473684210526)] plot (\x,{(\x)^2/2/0.2631578947368421});
		
		\draw[color=blue] [samples=50,rotate around={0:(0,0)},xshift=0cm,yshift=0cm,line width=6pt,domain=-7.2727272727272725:-2.5)] plot (\x,{(\x)^2/2/0.45454545454545453});
		\draw [samples=50,rotate around={0:(0,0)},xshift=0cm,yshift=0cm,line width=2pt,domain=-2:2)] plot (\x,{(\x)^2/2/0.14285714285714285});
		\draw [samples=50,rotate around={0:(0,0)},xshift=0cm,yshift=0cm,line width=2pt,domain=-2.5:2.87)] plot (\x,{(\x)^2/2/0.45454545454545453});
		\draw[color=blue] [samples=50,rotate around={0:(0,0)},xshift=0cm,yshift=0cm,line width=6pt,domain=2.87:7.2727272727272725)] plot (\x,{(\x)^2/2/0.45454545454545453});

		\draw [color=blue] [samples=50,rotate around={0:(0,0)},xshift=0cm,yshift=0cm,line width=6pt,domain=-10:-2.92)] plot (\x,{(\x)^2/2/0.7142857142857143});
		\draw [samples=50,rotate around={0:(0,0)},xshift=0cm,yshift=0cm,line width=2pt,domain=-2.92:2.92)] plot (\x,{(\x)^2/2/0.7142857142857143});
		\draw[color=blue] [samples=50,rotate around={0:(0,0)},xshift=0cm,yshift=0cm,line width=6pt,domain=2.92:10)] plot (\x,{(\x)^2/2/0.7142857142857143});

		\draw[color=blue] [samples=50,rotate around={0:(0,0)},xshift=0cm,yshift=0cm,line width=2pt,domain=-12:-1.5)] plot (\x,{(\x)^2/2/1});
		\draw [samples=50,rotate around={0:(0,0)},xshift=0cm,yshift=0cm,line width=2pt,domain=-1.5:2.47)] plot (\x,{(\x)^2/2/1});
		\draw[color=blue] [samples=50,rotate around={0:(0,0)},xshift=0cm,yshift=0cm,line width=2pt,domain=2.47:12)] plot (\x,{(\x)^2/2/1});
		\draw [line width=2pt] (0,-0.45) -- (0,14.21);
		\draw [line width=2pt] (-6.0,0) -- (6,0);
		\draw [line width=1pt] (0,1.1) -- (-5.5,1.1);
		\draw [line width=1pt] (0,3.05) -- (5.5,3.05);
		\draw [line width=1pt] (5.5,6) -- (-5.5,6);
		\draw [line width=1pt] (0,6.85) -- (-5.5,6.85);
		\draw [line width=1pt] (0,9.05) -- (5.5,9.05);
		\draw [line width=1pt] (5.5,12.05) -- (-5.5,12.05);
		
		\begin{scriptsize}
			\draw[color=black] (0.25,-0.25) node {$0$};
			\draw[color=black] (0.25,14.0) node {$\lambda$};
			\draw[color=black] (-5.8,-0.4) node {$E^+_k$};
			\draw[color=black] (5.8,-0.4) node {$E^-_k$};
			
			\draw[color=black] (-5.5,2.2) node {$\lambda=\lambda^+_1$};
			\draw[color=black] (-5.5,5.2) node {$\lambda=\lambda^+_2$};
			\draw[color=black] (-5.5,7.8) node {$\lambda=\lambda^+_3$};
			\draw[color=black] (-5.5,11.2) node {$\lambda=\lambda^+_4$};
			
			\draw[color=black] (5.5,2.2) node {$\lambda=\lambda^-_1$};
			\draw[color=black] (5.5,5.2) node {$\lambda=\lambda^-_2$};
			\draw[color=black] (5.5,8.2) node {$\lambda=\lambda^-_3$};
			\draw[color=black] (5.5,11.2) node {$\lambda=\lambda^-_4$};
			
		\end{scriptsize}
	\end{tikzpicture}
	\caption{Bifurcation phenomena  for \eqref{equilibria} with $p>2,$ $q<p$  and $A(z^+)=A(z^-)$. \underline{Description of the figures}: Solutions of \eqref{equilibria} associated to blue traces have a flat core. A thin blue trace represents a single solution with a flat core and thick blue trace represents a continuous of solutions with flat cores} \label{bifphe-qlessp}
\end{figure}
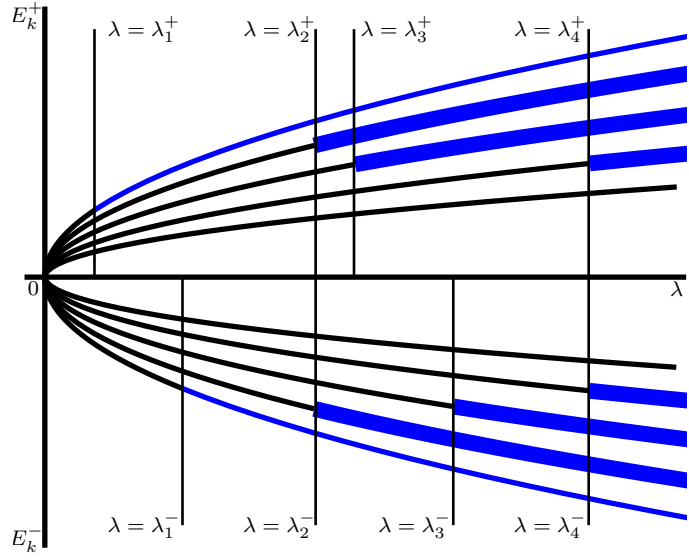

\bigskip

\begin{figure}
	\centering
	
	\begin{tikzpicture}[scale=0.6,rotate=-90]
		\clip(-6.26,-0.7) rectangle (8.0,14.21);
		
		\draw [color=blue][samples=50,rotate around={0:(-3,1)},xshift=-3cm,yshift=1cm,line width=2pt,domain=-4.347826086956522:-0.5)] plot (\x,{(\x)^2/2/0.2173913043478261});
		\draw [color=black][samples=50,rotate around={0:(-3,1)},xshift=-3cm,yshift=1cm,line width=2pt,domain=-0.5:4.347826086956522)] plot (\x,{(\x)^2/2/0.2173913043478261});

		\draw[color=blue] [samples=50,rotate around={0:(-3,3)},xshift=-3cm,yshift=3cm,line width=4.5pt,domain=-3.928571428571429:-0.5)] plot (\x,{(\x)^2/2/0.17857142857142858});
		\draw[color=black] [samples=50,rotate around={0:(-3,3)},xshift=-3cm,yshift=3cm,line width=2pt,domain=-0.5:3.928571428571429)] plot (\x,{(\x)^2/2/0.17857142857142858});
		
		\draw [color=blue][samples=50,rotate around={0:(-3,5)},xshift=-3cm,yshift=5cm,line width=4.5pt,domain=-3.25:-0.5)] plot (\x,{(\x)^2/2/0.125});
		\draw [color=black][samples=50,rotate around={0:(-3,5)},xshift=-3cm,yshift=5cm,line width=2pt,domain=-0.5:3.25)] plot (\x,{(\x)^2/2/0.125});

		\draw[color=blue] [samples=50,rotate around={0:(-3,7)},xshift=-3cm,yshift=7cm,line width=4.5pt,domain=-2.4285714285714284:-0.7)] plot (\x,{(\x)^2/2/0.07142857142857142});
		\draw[color=black] [samples=50,rotate around={0:(-3,7)},xshift=-3cm,yshift=7cm,line width=2pt,domain=-0.7:2.4285714285714284)] plot (\x,{(\x)^2/2/0.07142857142857142});

		\draw[color=black] [samples=50,rotate around={0:(-3,9)},xshift=-3cm,yshift=9cm,line width=2pt,domain=-1.8333333333333333:1.8333333333333333)] plot (\x,{(\x)^2/2/0.041666666666666664});
		
		\draw[color=black] [samples=50,rotate around={0:(-3,11)},xshift=-3cm,yshift=11cm,line width=2pt,domain=-1.4000000000000001:1.4000000000000001)] plot (\x,{(\x)^2/2/0.025});
		
		\draw[color=black] [samples=50,rotate around={0:(5,11)},xshift=5cm,yshift=11cm,line width=2pt,domain=-1.1:1.1)] plot (\x,{(\x)^2/2/0.025});
		
		\draw [samples=50,rotate around={0:(5,10)},xshift=5cm,yshift=10cm,line width=2pt,domain=-1.857142857142857:1.857142857142857)] plot (\x,{(\x)^2/2/0.07142857142857142});
		
		\draw[color=black] [samples=50,rotate around={0:(5,7)},xshift=5cm,yshift=7cm,line width=2pt,domain=-2.2:0.83)] plot (\x,{(\x)^2/2/0.1});
		\draw[color=blue] [samples=50,rotate around={0:(5,7)},xshift=5cm,yshift=7cm,line width=4.5pt,domain=0.83:2.2)] plot (\x,{(\x)^2/2/0.1});

		\draw [color=black][samples=50,rotate around={0:(5,6)},xshift=5cm,yshift=6cm,line width=2pt,domain=-3:0.7)] plot (\x,{(\x)^2/2/0.16666666666666666});
		\draw [color=blue][samples=50,rotate around={0:(5,6)},xshift=5cm,yshift=6cm,line width=4.5pt,domain=0.7:3)] plot (\x,{(\x)^2/2/0.16666666666666666});

		\draw [color=black][samples=50,rotate around={0:(5,3)},xshift=5cm,yshift=3cm,line width=2pt,domain=-2.857142857142857:0.5)] plot (\x,{(\x)^2/2/0.17857142857142858});
		\draw [color=blue][samples=50,rotate around={0:(5,3)},xshift=5cm,yshift=3cm,line width=4.5pt,domain=0.5:2.857142857142857)] plot (\x,{(\x)^2/2/0.17857142857142858});

		\draw[color=black] [samples=50,rotate around={0:(5,2)},xshift=5cm,yshift=2cm,line width=2pt,domain=-3.5:0.5)] plot (\x,{(\x)^2/2/0.25});
		\draw [color=blue] [samples=50,rotate around={0:(5,2)},xshift=5cm,yshift=2cm,line width=2pt,domain=0.5:3.5)] plot (\x,{(\x)^2/2/0.25});
		
		\draw [line width=2pt] (0,-0.45) -- (0,14.21);
		\draw [line width=1pt,domain=0.0:5.0] plot(\x,{(10-0*\x)/1});
		\draw [line width=1pt,domain=0.0:5.0] plot(\x,{(6-0*\x)/1});
		
		\draw [line width=1pt,domain=-3.0:0] plot(\x,{(9-0*\x)/1});
		\draw [line width=1pt,domain=-3.0:0] plot(\x,{(5-0*\x)/1});
		\draw [line width=1pt,domain=-3.0:5.0] plot(\x,{(11-0*\x)/1});
		\draw [line width=1pt,domain=-3.0:5.0] plot(\x,{(7-0*\x)/1});
		\draw [line width=1pt,domain=-3.0:5.0] plot(\x,{(3-0*\x)/1});
		\draw [line width=1pt,domain=0.0:5.0] plot(\x,{(2-0*\x)/1});
		\draw [line width=1pt,domain=-3.0:0] plot(\x,{(1-0*\x)/1});
		\draw [line width=2pt,domain=-6.26:8.0] plot(\x,{(-0-0*\x)/1});

		\draw [color=blue][line width=1pt] (0,1.54) -- (-5.6,1.54);
		\draw [color=blue][line width=1pt] (0,3.7) -- (-5.6,3.7);
		\draw [color=blue][line width=1pt] (0,6.) -- (-5.6,6.);
		\draw [color=blue][line width=1pt] (0,10.47) -- (-5.6,10.47);
		
		\draw [color=blue][line width=1pt] (0,2.52) -- (7.3,2.52);
		\draw [color=blue][line width=1pt] (0,3.7) -- (7.3,3.7);
		\draw [color=blue][line width=1pt] (0,7.49) -- (7.3,7.49);
		\draw [color=blue][line width=1pt] (0,10.47) -- (7.3,10.47);

		\begin{scriptsize}
			\draw[color=black] (-0.25,-0.25) node {$0$};
			\draw[color=black] (0.25,14.0) node {$\lambda$};
			\draw[color=black] (-5.9,-0.4) node {$E^+_k$};
			\draw[color=black] (7.7,-0.4) node {$E^-_k$};
			
			\draw[color=blue] (-5.9,1.54) node {$\lambda^+_{1}$};
			\draw[color=blue] (-5.9,3.7) node {$\lambda^+_{2}$};
			\draw[color=blue] (-5.9,6) node {$\lambda^+_{3}$};
			\draw[color=blue] (-5.9,10.47) node {$\lambda^+_{4}$};

			\draw[color=black] (0.4,0.95) node {$\lambda^+_{*,1}$};
			\draw[color=black] (-0.4,1.95) node {$\lambda^-_{*,1}$};
			\draw[color=black] (0.4,3.5) node {$\lambda^\pm_{*,2}$};
			\draw[color=black] (0.4,7.5) node {$\lambda^\pm_{*,4}$};
			\draw[color=black] (0.4,11.5) node {$\lambda^\pm_{*,6}$};
			\draw[color=black] (0.4,4.9) node {$\lambda^+_{*,3}$};
			\draw[color=black] (0.4,8.9) node {$\lambda^+_{*,5}$};
			\draw[color=black] (-0.4,5.9) node {$\lambda^-_{*,3}$};
			\draw[color=black] (-0.4,9.9) node {$\lambda^-_{*,5}$};
			
			\draw[color=blue] (7.6,2.52) node {$\lambda^-_{1}$};
			\draw[color=blue] (7.6,3.7) node {$\lambda^-_{2}$};
			\draw[color=blue] (7.6,7.49) node {$\lambda^-_{3}$};
			\draw[color=blue] (7.6,10.47) node {$\lambda^-_{4}$};
			
		\end{scriptsize}
	\end{tikzpicture}
	\caption{Bifurcation phenomena for \eqref{equilibria} case $q>p>2$ and $A(z^+)=A(z^-)$.} \label{bifphe-plessq-1}
\end{figure}

\bigskip

\begin{figure}
	\centering
	
	\begin{tikzpicture}[scale=0.6,rotate=-90]
		\clip(-6.26,-0.7) rectangle (8.0,14.21);
		
		\draw [color=blue][samples=50,rotate around={0:(-3,1)},xshift=-3cm,yshift=1cm,line width=2pt,domain=-4.347826086956522:-0.5)] plot (\x,{(\x)^2/2/0.2173913043478261});
		\draw [color=black][samples=50,rotate around={0:(-3,1)},xshift=-3cm,yshift=1cm,line width=2pt,domain=-0.5:4.347826086956522)] plot (\x,{(\x)^2/2/0.2173913043478261});

		\draw[color=blue] [samples=50,rotate around={0:(-3,3)},xshift=-3cm,yshift=3cm,line width=2pt,domain=-3.928571428571429:-0.5)] plot (\x,{(\x)^2/2/0.17857142857142858});
		\draw[color=black] [samples=50,rotate around={0:(-3,3)},xshift=-3cm,yshift=3cm,line width=2pt,domain=-0.5:3.928571428571429)] plot (\x,{(\x)^2/2/0.17857142857142858});
		
		\draw [color=blue][samples=50,rotate around={0:(-3,5)},xshift=-3cm,yshift=5cm,line width=2pt,domain=-3.25:-0.5)] plot (\x,{(\x)^2/2/0.125});
		\draw [color=black][samples=50,rotate around={0:(-3,5)},xshift=-3cm,yshift=5cm,line width=2pt,domain=-0.5:3.25)] plot (\x,{(\x)^2/2/0.125});

		\draw[color=blue] [samples=50,rotate around={0:(-3,7)},xshift=-3cm,yshift=7cm,line width=4.5pt,domain=-2.4285714285714284:-0.7)] plot (\x,{(\x)^2/2/0.07142857142857142});
		\draw[color=black] [samples=50,rotate around={0:(-3,7)},xshift=-3cm,yshift=7cm,line width=2pt,domain=-0.7:2.4285714285714284)] plot (\x,{(\x)^2/2/0.07142857142857142});

		\draw[color=black] [samples=50,rotate around={0:(-3,9)},xshift=-3cm,yshift=9cm,line width=2pt,domain=-1.8333333333333333:1.8333333333333333)] plot (\x,{(\x)^2/2/0.041666666666666664});
		
		\draw[color=black] [samples=50,rotate around={0:(-3,11)},xshift=-3cm,yshift=11cm,line width=2pt,domain=-1.4000000000000001:1.4000000000000001)] plot (\x,{(\x)^2/2/0.025});
		
		\draw[color=black] [samples=50,rotate around={0:(5,11)},xshift=5cm,yshift=11cm,line width=2pt,domain=-1.1:1.1)] plot (\x,{(\x)^2/2/0.025});
		
		\draw [samples=50,rotate around={0:(5,10)},xshift=5cm,yshift=10cm,line width=2pt,domain=-1.857142857142857:1.857142857142857)] plot (\x,{(\x)^2/2/0.07142857142857142});
		
		\draw[color=black] [samples=50,rotate around={0:(5,7)},xshift=5cm,yshift=7cm,line width=2pt,domain=-2.2:0.83)] plot (\x,{(\x)^2/2/0.1});
		\draw[color=blue] [samples=50,rotate around={0:(5,7)},xshift=5cm,yshift=7cm,line width=4.5pt,domain=0.83:2.2)] plot (\x,{(\x)^2/2/0.1});

		\draw [color=black][samples=50,rotate around={0:(5,6)},xshift=5cm,yshift=6cm,line width=2pt,domain=-3:0.7)] plot (\x,{(\x)^2/2/0.16666666666666666});
		\draw [color=blue][samples=50,rotate around={0:(5,6)},xshift=5cm,yshift=6cm,line width=4.5pt,domain=0.7:3)] plot (\x,{(\x)^2/2/0.16666666666666666});

		\draw [color=black][samples=50,rotate around={0:(5,3)},xshift=5cm,yshift=3cm,line width=2pt,domain=-2.857142857142857:0.5)] plot (\x,{(\x)^2/2/0.17857142857142858});
		\draw [color=blue][samples=50,rotate around={0:(5,3)},xshift=5cm,yshift=3cm,line width=2pt,domain=0.5:2.857142857142857)] plot (\x,{(\x)^2/2/0.17857142857142858});

		\draw[color=black] [samples=50,rotate around={0:(5,2)},xshift=5cm,yshift=2cm,line width=2pt,domain=-3.5:0.5)] plot (\x,{(\x)^2/2/0.25});
		\draw [color=blue] [samples=50,rotate around={0:(5,2)},xshift=5cm,yshift=2cm,line width=2pt,domain=0.5:3.5)] plot (\x,{(\x)^2/2/0.25});
		
		\draw [line width=2pt] (0,-0.45) -- (0,14.21);
		\draw [line width=1pt,domain=0.0:5.0] plot(\x,{(10-0*\x)/1});
		\draw [line width=1pt,domain=0.0:5.0] plot(\x,{(6-0*\x)/1});
		
		\draw [line width=1pt,domain=-3.0:0] plot(\x,{(9-0*\x)/1});
		\draw [line width=1pt,domain=-3.0:0] plot(\x,{(5-0*\x)/1});
		\draw [line width=1pt,domain=-3.0:5.0] plot(\x,{(11-0*\x)/1});
		\draw [line width=1pt,domain=-3.0:5.0] plot(\x,{(7-0*\x)/1});
		\draw [line width=1pt,domain=-3.0:5.0] plot(\x,{(3-0*\x)/1});
		\draw [line width=1pt,domain=0.0:5.0] plot(\x,{(2-0*\x)/1});
		\draw [line width=1pt,domain=-3.0:0] plot(\x,{(1-0*\x)/1});
		\draw [line width=2pt,domain=-6.26:8.0] plot(\x,{(-0-0*\x)/1});

		\draw [color=blue][line width=1pt] (0,1.54) -- (-5.6,1.54);
		\draw [color=blue][line width=1pt] (0,3.7) -- (-5.6,3.7);
		\draw [color=blue][line width=1pt] (0,6.) -- (-5.6,6.);
		\draw [color=blue][line width=1pt] (0,10.47) -- (-5.6,10.47);
		
		\draw [color=blue][line width=1pt] (0,2.52) -- (7.3,2.52);
		\draw [color=blue][line width=1pt] (0,3.7) -- (7.3,3.7);
		\draw [color=blue][line width=1pt] (0,7.49) -- (7.3,7.49);
		\draw [color=blue][line width=1pt] (0,10.47) -- (7.3,10.47);

		\begin{scriptsize}
			\draw[color=black] (-0.25,-0.25) node {$0$};
			\draw[color=black] (0.25,14.0) node {$\lambda$};
			\draw[color=black] (-5.9,-0.4) node {$E^+_k$};
			\draw[color=black] (7.7,-0.4) node {$E^-_k$};
			
			\draw[color=blue] (-5.9,1.54) node {$\lambda^+_{1}$};
			\draw[color=blue] (-5.9,3.7) node {$\lambda^+_{2}$};
			\draw[color=blue] (-5.9,6) node {$\lambda^+_{3}$};
			\draw[color=blue] (-5.9,10.47) node {$\lambda^+_{4}$};

			\draw[color=black] (0.4,0.95) node {$\lambda^+_{*,1}$};
			\draw[color=black] (-0.4,1.95) node {$\lambda^-_{*,1}$};
			\draw[color=black] (0.4,3.5) node {$\lambda^\pm_{*,2}$};
			\draw[color=black] (0.4,7.5) node {$\lambda^\pm_{*,4}$};
			\draw[color=black] (0.4,11.5) node {$\lambda^\pm_{*,6}$};
			\draw[color=black] (0.4,4.9) node {$\lambda^+_{*,3}$};
			\draw[color=black] (0.4,8.9) node {$\lambda^+_{*,5}$};
			\draw[color=black] (-0.4,5.9) node {$\lambda^-_{*,3}$};
			\draw[color=black] (-0.4,9.9) node {$\lambda^-_{*,5}$};
			
			\draw[color=blue] (7.6,2.52) node {$\lambda^-_{1}$};
			\draw[color=blue] (7.6,3.7) node {$\lambda^-_{2}$};
			\draw[color=blue] (7.6,7.49) node {$\lambda^-_{3}$};
			\draw[color=blue] (7.6,10.47) node {$\lambda^-_{4}$};
			
		\end{scriptsize}
	\end{tikzpicture}
	\caption{Bifurcation phenomena for \eqref{equilibria} case $q>p>2$ and $A(z^+)<A(z^-)$.} \label{bifphe-plessq-2}
\end{figure}

\bigskip

\begin{figure}
	\centering
	
	\begin{tikzpicture}[scale=0.6,rotate=-90]
		\clip(-6.26,-0.7) rectangle (8.0,14.21);
		
		\draw [color=blue][samples=50,rotate around={0:(-3,1)},xshift=-3cm,yshift=1cm,line width=2pt,domain=-4.347826086956522:-0.5)] plot (\x,{(\x)^2/2/0.2173913043478261});
		\draw [color=black][samples=50,rotate around={0:(-3,1)},xshift=-3cm,yshift=1cm,line width=2pt,domain=-0.5:4.347826086956522)] plot (\x,{(\x)^2/2/0.2173913043478261});

		\draw[color=blue] [samples=50,rotate around={0:(-3,3)},xshift=-3cm,yshift=3cm,line width=2pt,domain=-3.928571428571429:-0.5)] plot (\x,{(\x)^2/2/0.17857142857142858});
		\draw[color=black] [samples=50,rotate around={0:(-3,3)},xshift=-3cm,yshift=3cm,line width=2pt,domain=-0.5:3.928571428571429)] plot (\x,{(\x)^2/2/0.17857142857142858});
		
		\draw [color=blue][samples=50,rotate around={0:(-3,5)},xshift=-3cm,yshift=5cm,line width=4pt,domain=-3.25:-0.5)] plot (\x,{(\x)^2/2/0.125});
		\draw [color=black][samples=50,rotate around={0:(-3,5)},xshift=-3cm,yshift=5cm,line width=2pt,domain=-0.5:3.25)] plot (\x,{(\x)^2/2/0.125});

		\draw[color=blue] [samples=50,rotate around={0:(-3,7)},xshift=-3cm,yshift=7cm,line width=4.5pt,domain=-2.4285714285714284:-0.7)] plot (\x,{(\x)^2/2/0.07142857142857142});
		\draw[color=black] [samples=50,rotate around={0:(-3,7)},xshift=-3cm,yshift=7cm,line width=2pt,domain=-0.7:2.4285714285714284)] plot (\x,{(\x)^2/2/0.07142857142857142});

		\draw[color=black] [samples=50,rotate around={0:(-3,9)},xshift=-3cm,yshift=9cm,line width=2pt,domain=-1.8333333333333333:1.8333333333333333)] plot (\x,{(\x)^2/2/0.041666666666666664});
		
		\draw[color=black] [samples=50,rotate around={0:(-3,11)},xshift=-3cm,yshift=11cm,line width=2pt,domain=-1.4000000000000001:1.4000000000000001)] plot (\x,{(\x)^2/2/0.025});
		
		\draw[color=black] [samples=50,rotate around={0:(5,11)},xshift=5cm,yshift=11cm,line width=2pt,domain=-1.1:1.1)] plot (\x,{(\x)^2/2/0.025});
		
		\draw [samples=50,rotate around={0:(5,10)},xshift=5cm,yshift=10cm,line width=2pt,domain=-1.857142857142857:1.857142857142857)] plot (\x,{(\x)^2/2/0.07142857142857142});
		
		\draw[color=black] [samples=50,rotate around={0:(5,7)},xshift=5cm,yshift=7cm,line width=2pt,domain=-2.2:0.83)] plot (\x,{(\x)^2/2/0.1});
		\draw[color=blue] [samples=50,rotate around={0:(5,7)},xshift=5cm,yshift=7cm,line width=4.5pt,domain=0.83:2.2)] plot (\x,{(\x)^2/2/0.1});

		\draw [color=black][samples=50,rotate around={0:(5,6)},xshift=5cm,yshift=6cm,line width=2pt,domain=-3:0.7)] plot (\x,{(\x)^2/2/0.16666666666666666});
		\draw [color=blue][samples=50,rotate around={0:(5,6)},xshift=5cm,yshift=6cm,line width=2pt,domain=0.7:3)] plot (\x,{(\x)^2/2/0.16666666666666666});

		\draw [color=black][samples=50,rotate around={0:(5,3)},xshift=5cm,yshift=3cm,line width=2pt,domain=-2.857142857142857:0.5)] plot (\x,{(\x)^2/2/0.17857142857142858});
		\draw [color=blue][samples=50,rotate around={0:(5,3)},xshift=5cm,yshift=3cm,line width=2pt,domain=0.5:2.857142857142857)] plot (\x,{(\x)^2/2/0.17857142857142858});

		\draw[color=black] [samples=50,rotate around={0:(5,2)},xshift=5cm,yshift=2cm,line width=2pt,domain=-3.5:0.5)] plot (\x,{(\x)^2/2/0.25});
		\draw [color=blue] [samples=50,rotate around={0:(5,2)},xshift=5cm,yshift=2cm,line width=2pt,domain=0.5:3.5)] plot (\x,{(\x)^2/2/0.25});
		
		\draw [line width=2pt] (0,-0.45) -- (0,14.21);
		\draw [line width=1pt,domain=0.0:5.0] plot(\x,{(10-0*\x)/1});
		\draw [line width=1pt,domain=0.0:5.0] plot(\x,{(6-0*\x)/1});
		
		\draw [line width=1pt,domain=-3.0:0] plot(\x,{(9-0*\x)/1});
		\draw [line width=1pt,domain=-3.0:0] plot(\x,{(5-0*\x)/1});
		\draw [line width=1pt,domain=-3.0:5.0] plot(\x,{(11-0*\x)/1});
		\draw [line width=1pt,domain=-3.0:5.0] plot(\x,{(7-0*\x)/1});
		\draw [line width=1pt,domain=-3.0:5.0] plot(\x,{(3-0*\x)/1});
		\draw [line width=1pt,domain=0.0:5.0] plot(\x,{(2-0*\x)/1});
		\draw [line width=1pt,domain=-3.0:0] plot(\x,{(1-0*\x)/1});
		\draw [line width=2pt,domain=-6.26:8.0] plot(\x,{(-0-0*\x)/1});

		\draw [color=blue][line width=1pt] (0,1.54) -- (-5.6,1.54);
		\draw [color=blue][line width=1pt] (0,3.7) -- (-5.6,3.7);
		\draw [color=blue][line width=1pt] (0,6.) -- (-5.6,6.);
		\draw [color=blue][line width=1pt] (0,10.47) -- (-5.6,10.47);
		
		\draw [color=blue][line width=1pt] (0,2.52) -- (7.3,2.52);
		\draw [color=blue][line width=1pt] (0,3.7) -- (7.3,3.7);
		\draw [color=blue][line width=1pt] (0,7.49) -- (7.3,7.49);
		\draw [color=blue][line width=1pt] (0,10.47) -- (7.3,10.47);

		\begin{scriptsize}
			\draw[color=black] (-0.25,-0.25) node {$0$};
			\draw[color=black] (0.25,14.0) node {$\lambda$};
			\draw[color=black] (-5.9,-0.4) node {$E^+_k$};
			\draw[color=black] (7.7,-0.4) node {$E^-_k$};
			
			\draw[color=blue] (-5.9,1.54) node {$\lambda^+_{1}$};
			\draw[color=blue] (-5.9,3.7) node {$\lambda^+_{2}$};
			\draw[color=blue] (-5.9,6) node {$\lambda^+_{3}$};
			\draw[color=blue] (-5.9,10.47) node {$\lambda^+_{4}$};

			\draw[color=black] (0.4,0.95) node {$\lambda^+_{*,1}$};
			\draw[color=black] (-0.4,1.95) node {$\lambda^-_{*,1}$};
			\draw[color=black] (0.4,3.5) node {$\lambda^\pm_{*,2}$};
			\draw[color=black] (0.4,7.5) node {$\lambda^\pm_{*,4}$};
			\draw[color=black] (0.4,11.5) node {$\lambda^\pm_{*,6}$};
			\draw[color=black] (0.4,4.9) node {$\lambda^+_{*,3}$};
			\draw[color=black] (0.4,8.9) node {$\lambda^+_{*,5}$};
			\draw[color=black] (-0.4,5.9) node {$\lambda^-_{*,3}$};
			\draw[color=black] (-0.4,9.9) node {$\lambda^-_{*,5}$};
			
			\draw[color=blue] (7.6,2.52) node {$\lambda^-_{1}$};
			\draw[color=blue] (7.6,3.7) node {$\lambda^-_{2}$};
			\draw[color=blue] (7.6,7.49) node {$\lambda^-_{3}$};
			\draw[color=blue] (7.6,10.47) node {$\lambda^-_{4}$};
			
		\end{scriptsize}
	\end{tikzpicture}
	\caption{Bifurcation phenomena for \eqref{equilibria} case $q>p>2$ and $A(z^+)>A(z^-)$.} \label{bifphe-plessq-3}
\end{figure}

\bigskip

\begin{figure}
	\centering
	
	\begin{tikzpicture}[scale=0.6]
		\clip(-0.8,-6.0) rectangle (20.0,6.0);
		\draw [samples=50,rotate around={-90:(1,0)},xshift=1cm,yshift=0cm,line width=2.4pt,color=blue,domain=-3.92:-1.03)] plot (\x,{(\x)^2/2/0.5});
		\draw [samples=50,rotate around={-90:(1,0)},xshift=1cm,yshift=0cm,line width=2.4pt,domain=-1.03:1.2)] plot (\x,{(\x)^2/2/0.5});
		\draw [samples=50,rotate around={-90:(1,0)},xshift=1cm,yshift=0cm,line width=2.4pt,color=blue,domain=1.2:3.92)] plot (\x,{(\x)^2/2/0.5});

		\draw [samples=50,rotate around={-90:(3,0)},xshift=3cm,yshift=0cm,line width=5.0pt,color=blue,domain=-3:-1.15)] plot (\x,{(\x)^2/2/0.3333333333333333});
		\draw [samples=50,rotate around={-90:(3,0)},xshift=3cm,yshift=0cm,line width=2.4pt,domain=-1.15:1.15)] plot (\x,{(\x)^2/2/0.3333333333333333});
		\draw [samples=50,rotate around={-90:(3,0)},xshift=3cm,yshift=0cm,line width=5.0pt,color=blue,domain=1.15:3)] plot (\x,{(\x)^2/2/0.3333333333333333});
		
		\draw [samples=50,rotate around={-90:(6,0)},xshift=6cm,yshift=0cm,line width=5.0pt,color=blue,domain=-2.3:-1.42)] plot (\x,{(\x)^2/2/0.25});
		\draw [samples=50,rotate around={-90:(6,0)},xshift=6cm,yshift=0cm,line width=2.4pt,domain=-1.42:1.0)] plot (\x,{(\x)^2/2/0.25});
		\draw [samples=50,rotate around={-90:(6,0)},xshift=6cm,yshift=0cm,line width=5.0pt,color=blue,domain=1.0:2.3)] plot (\x,{(\x)^2/2/0.25});
		
		\draw [samples=50,rotate around={-90:(9,0)},xshift=9cm,yshift=0cm,line width=5.0pt,color=blue,domain=-1.6:-1.42)] plot (\x,{(\x)^2/2/0.16666666666666666});
		\draw [samples=50,rotate around={-90:(9,0)},xshift=9cm,yshift=0cm,line width=2.4pt,domain=-1.42:1.42)] plot (\x,{(\x)^2/2/0.16666666666666666});
		\draw [samples=50,rotate around={-90:(9,0)},xshift=9cm,yshift=0cm,line width=5.0pt,color=blue,domain=1.42:1.6)] plot (\x,{(\x)^2/2/0.16666666666666666});
		
		\draw [samples=50,rotate around={-90:(12,0)},xshift=12cm,yshift=0cm,line width=2.4pt,domain=-0.97:0.97)] plot (\x,{(\x)^2/2/0.1});
		
		\draw [line width=1pt] (2,0) -- (2,5.5);		
		\draw [line width=1pt] (5,0) -- (5,5.5);
		\draw [line width=1pt] (10,0) -- (10,5.5);
		\draw [line width=1pt] (15,0) -- (15,5.5);
		
		\draw [line width=1pt] (2.37,-5.5) -- (2.37,0);		
		\draw [line width=1pt] (5,-5.5) -- (5,0);
		\draw [line width=1pt] (8.05,-5.5) -- (8.05,0);
		\draw [line width=1pt] (15,-5.5) -- (15,0);
		
		\draw [line width=2.4pt,domain=-0.3:16.5] plot(\x,{(-0-0*\x)/1});
		\draw [line width=2.4pt] (0,-6.0) -- (0,6.0);
		\begin{scriptsize}
			\draw[color=black] (-0.4,-0.3) node {$0$};
			\draw[color=black] (16.4,-0.3) node {$\lambda$};
			\draw[color=black] (-0.5,5.5) node {$E^+_k$};
			\draw[color=black] (-0.5,-5.5) node {$E^-_k$};
			
			\draw[color=black] (0.8,-0.3) node {$\lambda_1$};
			\draw[color=black] (2.8,-0.3) node {$\lambda_2$};
			\draw[color=black] (5.8,-0.3) node {$\lambda_3$};
			\draw[color=black] (8.8,-0.3) node {$\lambda_4$};
			\draw[color=black] (11.8,-0.3) node {$\lambda_5$};
			
			\draw[color=black] (2.9,5.0) node {$\lambda=\lambda^+_1$};
			\draw[color=black] (5.9,5.0) node {$\lambda=\lambda^+_2$};
			\draw[color=black] (10.9,5.0) node {$\lambda=\lambda^+_3$};
			\draw[color=black] (15.9,5.0) node {$\lambda=\lambda^+_4$};
			
			\draw[color=black] (3.2,-5.0) node {$\lambda=\lambda^-_1$};
			\draw[color=black] (5.9,-5.0) node {$\lambda=\lambda^-_2$};
			\draw[color=black] (8.9,-5.0) node {$\lambda=\lambda^-_3$};
			\draw[color=black] (15.9,-5.0) node {$\lambda=\lambda^-_4$};
			
		\end{scriptsize}
	\end{tikzpicture}
	\caption{Bifurcation sequences  for \eqref{equilibria} with $2<q=p$ and $A(z^+)=A(z^-)$.
		}
	\label{bifphe-qequalp}
\end{figure}
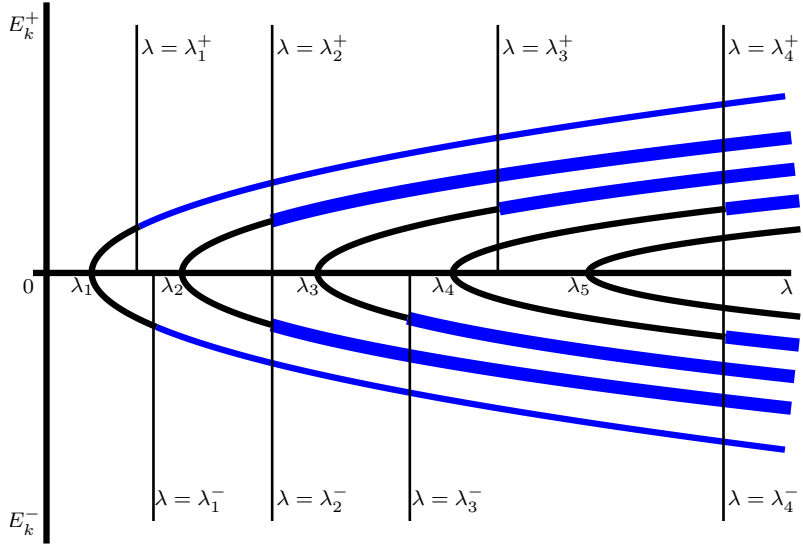

\bigskip
As an extension of the results of Theorem \ref{theorem-gv} (from $f$ odd to $f$ not necessarily odd) we have obtained the structure of  solutions of \eqref{equilibria}. 

\begin{itemize}

\item For $1<p\leqslant2$ and $q=p$ the structure of bifurcations is the same as that of Theorem \ref{theorem-gv}. Now, for $p>q$ and for each $\lambda>0$, the solution set of \eqref{equilibria} is formed by a countable (infinite) number solutions (no flat cores).


\item When $\lambda\geq \lambda_1$, $p>2$ and $p=q$, the set of solutions is formed by a finite number of connected components (each of them with a single equilibria or a continuous set of equilibria). When $\lambda>0$, $p>2$ and $p>q$ the set of solutions is formed by a countable (infinite) set of connected components (each of them with a single equilibria or a continuous set of equilibria).

\end{itemize}

\bigskip

 We observe that an interesting phenomena  occur, namely, the positive and negative solutions do not necessarily appear at the same value of the parameter (different from what is observed when $f$ is odd, \cite{Guedda-Veron}). Solutions belonging to $S^+_{2k+1}$ and $S^-_{2k+1}$ also do not necessarily appear at the same parameter value, however the  solutions that belong $S^\pm_{2k}$ appear at the same value of the parameter (similarly to what happens when $f$ odd, \cite{Guedda-Veron}), solutions with flat cores occur for the degenerate case (that is, $p>2$).

\bigskip

We also study the shape of the solutions of \eqref{equilibria}. Due to the fact that $f$ is not necessarily odd, as \cite{TAYA,Guedda-Veron}, the solutions do not satisfy the Chaffe-Infante's symmetry. 

For $p>2$ existence of solutions with flat cores will depend on the parameter $\lambda$ but their location may also on the relation between the areas $A(z^+)$ and  $A(z^-)$. That influences the shape of the solutions, which we briefly comment next. If $A(z^+)<A(z^-)$ then the flat cores are are located in the positive part of the solution only; if $A(z^+)>A(z^-)$ the flat cores are located in the negative part of the solution only; if $A(z^+)=A(z^-)$ flat cores are distributed between the positive and negative parts of the solution (this is the case obtained in \cite{TAYA,Guedda-Veron} all other cases are new in the literature).

\bigskip

Our results obtained for $f$ which is not odd are new in the literature that has the following interesting facts:
\begin{itemize}
	\item The positive and negative solutions of \eqref{equilibria} no longer appear together.
	\item The negative solution of \eqref{equilibria} is no longer the negative of the positive function.
	\item The positive and negative solutions still have the half interval symmetry.
	\item There are two positive (negative) solutions since $q>p$
	\item Oscillating solutions do not have the half interval symmetry.
	\item The location of the Flat Cores depend on the the relation between the areas $A(z^+)$ and $A(z^-)$.
\end{itemize}

\bigskip

The Time-map method used in \cite{Bcv,Carvalho et al.,CLR-PAMS,CH-IN,Chaf-Inf,Ha,HE,Guedda-Veron,LICALUMO,TAYA}, based on the analysis of the phase diagram associated with the elliptic problem, is commonly used in the literature to determine bifurcation sequences. This method is improved by Takeuchi and Yamada in \cite{TAYA} by reducing the time-map method to the study of the behavior of the integral functions $I$ and $J$ defined in \eqref{I} and \eqref{J} respectively. The shape of the function $I$ and $J$ is determined according to the relationship between the numbers $q$ and $p$ (which is why we separate the study to these three cases). In the case when $f$ is not odd it is still possible to adapt this time-map method to study the solutions of \eqref{equilibria}. The strategy in this case is to separate the solutions that start with a positive slope at $0$ from those that begin with a negative slope at $0$.

\bigskip

This paper is organized as follows. In Section \ref{S-1} we construct explicitly the solution to equation \eqref{equilibria}. Section \ref{S-pos-neg} is dedicated to study of existence of positive and negative solutions. In Section \ref{S-even} we determine the structure of  solutions that live in $S^\pm_{2j}$. In Section \ref{S-odd} we determine the structure of  solutions that live in $S^\pm_{2j+1}$.  And finally in Section \ref{S-reg} we give a optimal regularity result for the solutions of \eqref{equilibria} (including, in some particular situations, $C^2$ regularity at points with zero derivative).

\section{The construction of the solutions}\label{S-1}

This section is dedicated to the construction of the solutions of \eqref{equilibria}. From \eqref{equilibria}, it is clear that a solution of \eqref{equilibria} must satisfy
$$
\int_0^1 |\phi_x|^p dx = \lambda \left(\int_0^1 |\phi|^q dx-\int_0^1 \phi f(\phi) dx\right).
$$
A non-trivial solution for \eqref{equilibria} only exist if (see \cite{otani})
$$
\lambda>\inf_{u\neq 0,u\in W^{1,p}_0(0,1)}\frac{\int^1_0|u_x(x)|^pdx}{\int^1_0|u(x)|^qdx}.
$$ 

\bigskip

Note that, multiplying the first equation in \eqref{equilibria} by $\phi_x$ and integrating from $a$ to $x\in (a,1)$ we obtain
\begin{equation}\label{energy3}
|\phi_x(x)|^p=|\phi_x(a)|^p+\frac{\lambda p}{(p-1)q} \left( qF(\phi(x))-qF(\phi(a))-|\phi(x)|^q+|\phi(a)|^q\right).
\end{equation}
Now, if $a=0$ and $\phi(0)=\phi_x(0)=0$ we have that
$$
|\phi_x(x)|^p=\frac{\lambda p}{(p-1)q}( qF(\phi(x))-|\phi(x)|^q)
$$
and since $t\mapsto qF(t)-|t|^q$ is negative for $0<|t|<\rho$, for some $\rho>0$ we must have that $\phi_x$ is identically zero and that $\phi\equiv 0$.

\bigskip

Now we start the construction of the solution $\phi\in S^+_j$ (see \eqref{set-zeros} for the definition of $S^+_j$)  of \eqref{equilibria}, for that we assume that $\phi_x(0)=r>0$ and $\phi(0)=0$ in the energy relation \eqref{energy3}. It follows that $\phi$ is increasing in some interval $[0,x_0]$ and from \eqref{energy3} we have
$$
\phi_x(x) = \left(r^p+\frac{\lambda p}{(p-1)q}(qF(\phi(x))-\phi(x)^q)\right)^\frac{1}{p}.
$$
Hence, for $x\in [0,x_0]$,
$$
x=\int_0^{\phi(x)}  \left(r^p+\frac{\lambda p}{(p-1)q}(qF(t)-t^q)\right)^{-\frac{1}{p}}dt.
$$
Moreover this formula remains valid as long as $\phi(x)$ is smaller than the first positive zero of the function $s\mapsto \psi(r,s)=r^p+\frac{\lambda p}{(p-1)q}(qF(s)-s^q)$. Now
\begin{equation*}
\frac{d}{ds}\psi(r,s)=\frac{\lambda p}{(p-1)}(f(s)-s^{q-1})\left\{
\begin{split} 
&\!<0,\ \hbox{ for } s\in (0,z^+),\\
&\!=0,\ \hbox{ for } s=z^+. \\
\end{split}\right.
\end{equation*}

\bigskip

If $r>r(\lambda)$ (see \eqref{max-inclination-positive} for the definition of $r(\lambda)$)
$$
\phi_x(x) = \left(r^p+\frac{\lambda p}{(p-1)q}(qF(\phi(x))-\phi(x)^q)\right)^\frac{1}{p}\geqslant \left(r^p-r(\lambda)^p\right)^\frac{1}{p}>0
$$
for all $x\in (0,\infty)$ and $\phi$ is not a solution for \eqref{equilibria}. If $r= r(\lambda)$, then $\psi(r,s)=r^p+\frac{\lambda p}{(p-1)q}(qF(s)-s^q)$ has a double zero in $z^+$. Consequently, if $1<p\leqslant 2$,
$$
x(s)=\int_0^{s}  \left(r(\lambda)^p+\frac{\lambda p}{(p-1)q}(qF(t)-t^q)\right)^{-\frac{1}{p}}dt\stackrel{s\to z^+}{\longrightarrow} +\infty
$$
and $0<\phi_x(x(s))\stackrel{s\to z^+}{\longrightarrow}0$. As a consequence of that $\phi$ is not a solution for \eqref{equilibria}. Now, if $2<p<\infty$, then
\begin{equation}\label{xlambda}
x(\lambda):=\int_0^{z^+}  \left(r(\lambda)^p+\frac{\lambda p}{(p-1)q}(qF(t)-t^q)\right)^{-\frac{1}{p}}dt<\infty,
\end{equation}
$\phi(x(\lambda))=z^+$ and $\phi_x(z^+)=0$. 

\bigskip

It remains to consider the case $r< r(\lambda)$. In this case $s\mapsto \psi(r,s)$ admits a simple zero $z(r)$ in $(0,z^+)$ given by
\begin{equation}\label{zero}
r^p +\frac{\lambda p}{(p-1)q} \left( qF(z(r))-z(r)^q\right)=0.
\end{equation}
Let
\begin{equation}\label{theta}
\theta(r)=\int_0^{z(r)} (\psi(r,t))^{-\frac{1}{p}}dt
\end{equation}
and $\phi(\theta(r))=z(r)$, $\phi_x(\theta(r))=(\psi(r,\phi(\theta(r))))^{\frac{1}{p}}=(\psi(r,z(r)))^{\frac{1}{p}}=0$. From \eqref{energy3} for $a=\theta(r)$ and $x\in [\theta(r),\Theta] $ for some $\Theta>\theta(r)$
\begin{equation*}
|\phi_x(x)|^p =\frac{\lambda p}{(p-1)q} \left( qF(\phi(x))-qF(z(r))-|\phi(x)|^q+z(r)^q\right)
\end{equation*}
and from \eqref{zero}
\begin{equation*}
|\phi_x(x)|^p =r^p+\frac{\lambda p}{(p-1)q} \left( qF(\phi(x))-|\phi(x)|^q\right)
\end{equation*}
and $\phi$ is decreasing in $[\theta(r),\Theta]$, from which we have
\begin{equation*}
\begin{split}
\phi_x(x) =&-\left(r^p+\frac{\lambda p}{(p-1)q} \left( qF(\phi(x))-|\phi(x)|^q\right)\right)^\frac{1}{p}dt.
\end{split}
\end{equation*}
Therefore, 
$$
x-\theta(r)=-\int_{z(r)}^{\phi(x)} \left(r^p+\frac{\lambda p}{(p-1)q} \left( qF(t)-|t|^q\right)\right)^{-\frac{1}{p}}dt
$$
this formula will holds for as long as $\phi$ is decreasing. In particular, the formula above holds as long as $\phi>0$ and, in that case,
$$
x-\theta(r)=-\int_{z(r)}^{\phi(x)} \left(r^p+\frac{\lambda p}{(p-1)q} \left( qF(t)-t^q\right)\right)^{-\frac{1}{p}}dt.
$$
Hence, for $x_1\in (0,\theta(r))$, making $x_2=2 \theta(r)-x_1$, we have that $x_2-\theta(r)=\theta(r)-x_1$,
\begin{equation*}
\begin{split}
x_2-\theta(r)&=-\int_{z(r)}^{\phi(x_2)} \left(r^p+\frac{\lambda p}{(p-1)q} \left( qF(t)-t^q\right)\right)^{-\frac{1}{p}}dt\\
&=\theta(r)-x_1=
\int_{\phi(x_1)}^{z(r)}  \left(r^p+\frac{\lambda p}{(p-1)q} \left( qF(t)-t^q\right)\right)^{-\frac{1}{p}}dt
\end{split}
\end{equation*}
and $\phi(x_1)=\phi(x_2)$. Therefore, $\phi(\theta(r)-x)=\phi(\theta(r)+x)$, $x\in (0,\theta(r))$. Now $\phi(2\theta(r))=\phi(0)=0$ and $\phi_x(2\theta(r))=-\phi_x(0)=-r<0$. Consequently,  $\phi$ changes sign in $x=2\theta(r)$, is decreasing  in some interval $[2\theta(r),y_0]$ and $\phi_x$ will vanish again. From \eqref{energy3} with $a=2\theta(r)$ 
$$
\phi_x(x) = -\left(r^p+\frac{\lambda p}{(p-1)q}(qF(\phi(x))-|\phi(x)|^q)\right)^\frac{1}{p}
$$
then, for $x\in [2\theta(r),y_0]$,
$$
x-2\theta(r)=\int^0_{\phi(x)}  \left(r^p+\frac{\lambda p}{(p-1)q}(qF(t)-|t|^q)\right)^{-\frac{1}{p}}dt.
$$
Moreover this formula remains valid as long as $\phi(x)$ is smaller than the first negative zero of the function $s\mapsto \psi^-(r,s)=r^p+\frac{\lambda p}{(p-1)q}(qF(s)-|s|^q)$. Now
\begin{equation*}
\frac{d}{ds}\psi^-(r,s)=\frac{\lambda p}{(p-1)}(f(s)+|s|^{q-1})\left\{
\begin{split} 
&\!>0,\ \hbox{ for } s\in (z^-,0),\\
&\!=0,\ \hbox{ for } s=z^-. \\
\end{split}\right.
\end{equation*}

\bigskip

 If $r> r^-(\lambda)$ (see  \eqref{max-inclination-negative} for the definition of $r^-(\lambda)$)
$$
\phi_x(x) = -\left(r^p+\frac{\lambda p}{(p-1)q}(qF(\phi(x))-|\phi(x)|^q)\right)^\frac{1}{p}< -\left(r^p-r^-(\lambda)^p\right)^\frac{1}{p}<0.
$$
for all $x\in (0,\infty)$ and $\phi$ is not a solution for \eqref{equilibria}. If $r= r^-(\lambda)$, then $\psi^-(r,s)=r^p+\frac{\lambda p}{(p-1)q}(qF(s)-|s|^q)$ has a double zero in $z^-$. Consequently, if $1<p\leqslant 2$,
$$
y(s)=\int^0_{s}  \left(r(\lambda)^p+\frac{\lambda p}{(p-1)q}(qF(t)-|t|^q)\right)^{-\frac{1}{p}}dt\stackrel{s\to z^-}{\longrightarrow} +\infty
$$
and $0<\phi_x(y(s))\stackrel{s\to z^-}{\longrightarrow}0$. As a consequence of that $\phi$ is not a solution for \eqref{equilibria}. Now, if $2<p<\infty$, then
\begin{equation}\label{xlambda}
y(\lambda):=\int^0_{z^-}  \left(r^-(\lambda)^p+\frac{\lambda p}{(p-1)q}(qF(t)-|t|^q)\right)^{-\frac{1}{p}}dt<\infty. 
\end{equation}
and $\phi(y(\lambda))=z^-$ and $\phi_x(z^-)=0$. 

\bigskip

Again it remains to consider the case $r< r^-(\lambda)$. In this case $s\mapsto \psi^-(r,s)$ admits a simple zero $S(r)$ in $(z^-,0)$ given by

\begin{equation}\label{negativezero}
r^p +\frac{\lambda p}{(p-1)q} \left( qF(S(r))-|S(r)|^q\right)=0.
\end{equation}
Let
\begin{equation}\label{alpha}
\alpha(r)=\int^0_{S(r)} (\psi^-(r,t))^{-\frac{1}{p}}dt
\end{equation}
and $\phi(2\theta(r)+\alpha(r))=S(r)$, $\phi_x(2\theta(r)+\alpha(r))=-(\psi^-(r,\phi(2\theta(r)+\alpha(r))))^{\frac{1}{p}}=-(\psi^-(r,S(r)))^{\frac{1}{p}}=0$. Using the energy relation \eqref{energy3} for $a=2\theta(r)+\alpha(r)$ and $x\in [2\theta(r)+\alpha(r)+,\Theta] $, for some $\Theta>2\theta(r)+\alpha(r)$,
\begin{equation*}
|\phi_x(x)|^p =\frac{\lambda p}{(p-1)q} \left( qF(\phi(x))-qF(S(r))-|\phi(x)|^q+|S(r)|^q\right).
\end{equation*}
Now, from \eqref{negativezero},
\begin{equation*}
|\phi_x(x)|^p =r^p+\frac{\lambda p}{(p-1)q} \left( qF(\phi(x))-|\phi(x)|^q\right)
\end{equation*}
and $\phi$ is increasing in $[2\theta(r)+\alpha(r),\Theta]$. From this we have 
\begin{equation*}
\begin{split}
\phi_x(x) =&\left(r^p+\frac{\lambda p}{(p-1)q} \left( qF(\phi(x))-|\phi(x)|^q\right)\right)^\frac{1}{p}.
\end{split}
\end{equation*}
Therefore, 
$$
x-2\theta(r)-\alpha(r)=\int_{S(r)}^{\phi(x)} \left(r^p+\frac{\lambda p}{(p-1)q} \left( qF(t)-|t|^q\right)\right)^{-\frac{1}{p}}dt
$$
this last formula will holds for as long as $\phi$ is increasing. In particular, the formula above holds as long as $\phi<0$ and hence, for $x_3\in (2\theta(r),2\theta(r)+\alpha(r))$, making $x_4=2\theta(r)+2\alpha(r)-x_3$, we have that $x_4-2\theta(r)-\alpha(r)=2\theta(r)+\alpha(r)-x_3$ and
\begin{equation*}
	\begin{split}
		\int_{S(r)}^{\phi(x_3)} \left(\psi^-(r,t)\right)^{-\frac{1}{p}}dt =&\int_{0}^{\phi(x_3)} \left(\psi^-(r,t)\right)^{-\frac{1}{p}}dt+\int_{S(r)}^{0} \left(\psi^-(r,t)\right)^{-\frac{1}{p}}dt\\
		=&-\int^{0}_{\phi(x_3)} \left(\psi^-(r,t)\right)^{-\frac{1}{p}}dt+\int_{S(r)}^{0} \left(\psi^-(r,t)\right)^{-\frac{1}{p}}dt\\
		=&-(x_3-2\theta(r))+\alpha(r)\\
		=&x_4-2\theta(r)-\alpha(r)\\
		=&\int_{S(r)}^{\phi(x_4)} \left(\psi^-(r,t)\right)^{-\frac{1}{p}}dt
			\end{split}
\end{equation*}
and $\phi(x_3)=\phi(x_4)$. Therefore, $\phi(2\theta(r)+\alpha(r)-x)=\phi(2\theta(r)+\alpha(r)+x)$, $x\in (0,\alpha(r))$. Now $\phi(2\theta(r)+2\alpha(r))=\phi(2\theta(r))=0$ and $\phi_x(2\theta(r)+2\alpha(r))=-\phi_x(2\theta(r))=r$ and we may proceed with this analysis and to conclude the following necessary and sufficient conditions 
\begin{itemize}
\item $\phi\in S^+_1$ is a solution of \eqref{equilibria} if and only if $2\theta(r)=1$,
\item $\phi\in S^+_{2j-1}$ is a solution of  \eqref{equilibria} if and only if $2j\theta(r)+2(j-1)\alpha(r)=1$, for $j=2,\cdots$
\item $\phi\in S^+_{2j}$ is a solution of  \eqref{equilibria} if and only if  $2j(\theta(r)+\alpha(r))=1$, for $j=1,2,\cdots$
\end{itemize}
similarly we conclude that the necessary and sufficient condition for $\phi\in S^-_j$ to be a
\begin{itemize}
	\item $\phi\in S^-_1$ is a solution of \eqref{equilibria} if and only if $2\alpha(r)=1$,
	\item $\phi\in S^-_{2j-1}$ is a solution of  \eqref{equilibria} if and only if $2j\alpha(r)+2(j-1)\theta(r)=1$, for $j=2,3\cdots,$
	\item $\phi\in S^-_{2j}$ is a solution of  \eqref{equilibria} if and only if  $2j(\alpha(r)+\theta(r))=1$, for $j=1,2,\cdots$.
\end{itemize}
  
From this, it becomes clear that we must study the image of the functions 
\begin{itemize}
	\item $\theta:(0,r(\lambda))\to \R^+$ to determine the positive solution of \eqref{equilibria},
	\item $\alpha:(0,r^-(\lambda))\to \R^+$ to determine the negative solution of \eqref{equilibria} and
	\item $\theta, \alpha:(0,r^*(\lambda)=\min\{r(\lambda),r^-(\lambda)\})\to \R^+$ to determine the others solutions of \eqref{equilibria}.
\end{itemize}

\section{Positive and negative solutions}\label{S-pos-neg}
	
From the considerations in the previous section, to ensure the existence of positive and negative solutions of \eqref{equilibria} we need study when $1\in Im\left(2\theta\right)$ on  $(0,r(\lambda))$ and $1\in Im\left(2\alpha\right)$ on $(0,r^-(\lambda))$ respectively. We begin with the positive one, first note that the function $z:(0,r(\lambda))\to (0,z^+)$ is given by \eqref{zero}. Now, differentiating implicitly \eqref{zero} with respect to $r$, we obtain 
\begin{equation}\label{zderivate}
z'(r)=\frac{ (p-1) r^{p-1}}{ \lambda \left((z(r))^{q-1}-f(z(r))\right)}>0, \ z(r)\in (0,z^+)
\end{equation}
that implies  $z$ is strictly increasing on $(0,r(\lambda))$.
 
\bigskip

From  \eqref{zero} and \eqref{theta} we obtain  
\begin{equation}\label{thetaI}
\theta(r)=\left(\frac{p-1}{\lambda p}\right)^{\frac{1}{p}}\int_0^{z(r)}\left(F(t)-F(z(r))+\frac{z(r)^q-t^q}{q}\right)^{-\frac{1}{p}}dt=\left(\frac{p-1}{\lambda p}\right)^{\frac{1}{p}}I(z(r)),
\end{equation}
where the function $I$ is defined by
\begin{equation}\label{I}
I(a)=\int_0^{a}\left(F(t)-F(a)+\frac{a^q-t^q}{q}\right)^{-\frac{1}{p}}dt
\end{equation}
for all $a\in(0,z^+)$.
\begin{lemma}\label{function_I} 
The function $I(\cdot)$ has the following properties
\begin{enumerate}
\item[i)] $I$ is continuous in $(0,z^+)$
\item[ii)]  \begin{equation*}
\lim_{a\rightarrow 0^+} I(a)=
\begin{cases} 
0,&\ \hbox{ if } q<p,\\
p^{\frac{1}{p}}\int^1_0(1-y^p)^{-\frac{1}{p}}dy,&\ \hbox{ if } q=p,\\
+\infty,&\ \hbox{ if } q>p
\end{cases}
\end{equation*}
\item[iii)]    
\begin{equation*}
	I(z^+)=\begin{cases} 
		+\infty,&\ \hbox{ if } 1<p\leqslant2,\\
		<+\infty,&\ \hbox{ if } p>2.
	\end{cases}
\end{equation*}
\item[iv)]
\begin{equation*}
\begin{cases} 
\text{$I$ is increasing in $(0,z^+)$,}&\ \hbox{ if } q\leqslant p,\\
\text{$I(a)\geqslant I(a_*)>0$ for all $a\in(0,z^+)$ for some $a_*\in(0,z^+)$,}&\ \hbox{ if } q>p.
\end{cases}
\end{equation*}
\end{enumerate}
\end{lemma}

\begin{proof}
By the change of variables $t=ay$, we have
\begin{equation*}
I(a)=\int^1_0\left(\Phi(a,y)\right)^{-\frac{1}{p}}dy.
\end{equation*}
where 
\begin{equation}\label{Phi}
\Phi(a,y):=\frac{F(ay)-F(a)}{a^p}+a^{q-p}\frac{1-y^q}{q}>0
\end{equation}
for all $a\in(0,z^+)$ and $y\in(0,1)$. Using the hypothesis that $\mu\mapsto \frac{f(\mu)}{\mu^{q-1}}$ is strictly increasing on $(0,z^+)$ we have the estimate  
\begin{equation}\label{Fsupestimative}
F(a)-F(ay)=\int^{a}_{ay}f(\mu)d\mu<\int_{ay}^{a}\left(\frac{\mu}{a}\right)^{q-1}f(a)d\mu=\frac{f(a)a}{q}(1-y^q)
\end{equation}
This  inequality imply
\begin{equation*}
0<\frac{a^{q-1}-f(a)}{qa^{p-1}}(1-y^q)<\Phi(a,y)<\frac{a^{q-p}}{q}(1-y^q)
\end{equation*}
that implies
\begin{equation*}
	\frac{q^{\frac{1}{p}}}{a^{\frac{q-p}{p}}}\int^1_0(1-y^q)^{-\frac{1}{p}}dy<I(a)<\frac{(qa^{p-1})^{\frac{1}{p}}}{(a^{q-1}-f(a))^{\frac{1}{p}}}\int^1_0(1-y^q)^{-\frac{1}{p}}dy.
\end{equation*}
This last inequality  we obtain $i)$ and $ii)$. 

\bigskip

By Taylor's formula, we have  $F(t)-F(z^+)+\frac{(z^+)^q-t^q}{q}=-\frac{L(q-1)(z^+)^{q-2}}{2}(t-z^+)^2+h((z^+-t)^2)$ (with $\displaystyle\lim_{t\to z^+}\frac{h((z^+-t)^2)}{(z^+-t)^2}=0$) for $t$ close to $z^+$, since $z^+$ is double zero of the function $t\mapsto F(t)-F(z^+)+\frac{(z^+)^q-t^q}{q}$ and this function is twice differentiable in $z^+$(consequence of \eqref{hypothesis-zero}). From \eqref{I}
\begin{align*}
I(z^+)&=\int_0^{a}\left(F(t)-F(z^+)+\frac{(z^+)^q-t^q}{q}\right)^{-\frac{1}{p}}dt+\int_a^{z^+}\left(F(t)-F(z^+)+\frac{(z^+)^q-t^q}{q}\right)^{-\frac{1}{p}}dt\\
&\approx\int_0^{a}\left(F(t)-F(z^+)+\frac{(z^+)^q-t^q}{q}\right)^{-\frac{1}{p}}dt+\left(-\frac{L(q-1)(z^+)^{q-2}}{2}\right)^{-\frac{1}{p}}\int_a^{z^+}\left(z^+-t\right)^{-\frac{2}{p}}dt\\
\end{align*}
for some $a$ close to $z^+$, where $L$ is defined in \eqref{hypothesis-zero}. So,  we have  $iii)$.

\bigskip

Now, we will show $iv)$. In fact, we differentiate $I$, with respect to $a$, we obtain
\begin{equation}\label{Iderivate}
I'(a)=-\frac{1}{p}\int^1_0\left(\Phi(a,y)\right)^{-\frac{1}{p}-1}\Phi_a(a,y)dy
\end{equation}
where
\begin{equation*}
\Phi_a(a,y)=\frac{q\left[f(ay)ay-f(a)a-p(F(ay)-F(a))\right]+(q-p)a^{q}(1-y^q)}{qa^{p+1}}.
\end{equation*}
We need to study the sign of $\Phi_a(a,y)$. Let's go with that, using \eqref{Fsupestimative} and $\mu\mapsto \frac{f(\mu)}{\mu^{q-1}}$ is strictly increasing on $(0,z^+)$, we obtain
\begin{equation}\label{Phiderivatesupestimate}
\Phi_a(a,y)<\frac{(q-p)(a^{q-1}-f(a))(1-y^q)}{qa^p}
\end{equation}
So, for $q\leqslant p$, $\Phi_a(a,y)<0$ for all $(a,y)\in(0,z^+)\times(0,1)$, then from \eqref{Iderivate} we conclude that $I$ is strictly increasing since $q\leqslant p$. Moreover, from \eqref{Phiderivatesupestimate} we have $\Phi_a(z^+,y)<0$ since $(z^+)^{q-1}-f(z^+)=0$. Thus  there exists $a_*\in(0,z^+)$ such that $\Phi_a(a_*,y)=0$ for all $y\in(0,1)$, that show $iv)$.
\end{proof}

\bigskip

Now, we will continue with our study of the function $2\theta$, Lemma \ref{function_I} gives
\begin{equation}\label{limitstheta}
\begin{cases}
\text{$\theta$ is a increasing function and}\quad\frac{1}{2}\left(\frac{\lambda_1}{\lambda}\right)^\frac{1}{p}\stackrel{r\to 0^+}{\longleftarrow}\theta(r)\stackrel{r\to r(\lambda)^-}\longrightarrow \theta(r(\lambda)),&\text{if $q=p$},\\
\text{$\theta$ is a increasing function and}\quad 0\stackrel{r\to 0^+}{\longleftarrow}\theta(r)\stackrel{r\to r(\lambda)^-}\longrightarrow \theta(r(\lambda)),&\text{if $q<p$},\\
\text{$\theta(r)\geqslant\left(\frac{p-1}{\lambda p}\right)^{\frac{1}{p}}I(a_*)$ for all $r\in(0,r(\lambda))$ and}\quad+\infty\stackrel{r\to 0^+}{\longleftarrow}\theta(r)\stackrel{r\to r(\lambda)^-}{\longrightarrow} \theta(r(\lambda)),&\text{if $q>p$},
\end{cases}
\end{equation}
 where
\begin{equation*}
\theta(r(\lambda)) = \int_0^{z^+} \left(r(\lambda)^p+\frac{\lambda p}{(p-1)q}(qF(t)- t^q)\right)^{-\frac{1}{p}}dt=
\begin{cases}
+\infty,& \quad \text{if} \quad1<p\leqslant 2,\\
x(\lambda)<+\infty,& \quad \text{if} \quad2<p<\infty,
\end{cases}
\end{equation*} 
since that $z(r)\stackrel{r\to 0^+}{\longrightarrow} 0$  and $z(r)\stackrel{r\to r(\lambda)^-}\longrightarrow z^+.$ 

\bigskip

First we suppose $1<p\leqslant 2$. Then we have 
\begin{equation*}
2\theta((0,r(\lambda)))=
\begin{cases}
\left(\left(\frac{\lambda_1}{\lambda}\right)^\frac{1}{p},+\infty\right),&\text{if $q=p$}\\
\left(0,+\infty\right),&\text{if $q<p$}\\
\left[2\left(\frac{p-1}{\lambda p}\right)^\frac{1}{p}I(a_*),+\infty\right),&\text{if $q>p$.}\\
\end{cases}
\end{equation*}
If $q=p$ and $\lambda_1 <\lambda$ we have that $\left(\frac{\lambda_1}{\lambda}\right)^\frac{1}{p}<1$ and there is a unique $r_1\in (0,r(\lambda))$ such that $2\theta(r_1)=1$ and, associated to that, exactly one $\phi\in S_1^+$ positive solution of \eqref{equilibria}. Now, if $q<p$ and $0<\lambda$ there is $r_1\in (0,r(\lambda))$ such that $2\theta(r_1)=1$ and, associated to that, exactly one $\phi\in S_1^+$ positive solution of \eqref{equilibria}. It remains the case  $q>p$, if $\lambda^+_{*,1}:=\frac{p-1}{p}\left(2I(a_*)\right)^p<\lambda$  there are at least $r_1,r^*_1\in (0,r(\lambda))$ such that $r_1<r^*_1$ and  $2\theta(r_1)=2\theta(r^*_1)=1$ and, associated to those, exactly two $\phi\in S_1^+$ positive solutions of \eqref{equilibria}. And if $\frac{p-1}{p}\left(2I(a_*)\right)^p=\lambda$ there is at least $r_1\in (0,r(\lambda))$ such that $2\theta(r_1)=1$ and, associated to that, exactly one $\phi\in S_1^+$ positive solution of \eqref{equilibria}.

\bigskip

Now, we suppose $p>2$, then we have
\begin{equation*}
2\theta((0,r(\lambda)])=
\begin{cases}
\left(\left(\frac{\lambda_1}{\lambda}\right)^\frac{1}{p},2x(\lambda)\right],&\text{if $q=p$}\\
\left(0,2x(\lambda)\right],&\text{if $q<p$}\\
\left[2\left(\frac{p-1}{\lambda p}\right)^\frac{1}{p}I(a_*),2x(\lambda),\right],&\text{if $q>p$}\\
\end{cases}
\end{equation*}
If $q=p$ and $\lambda_1 <\lambda$ and $2x(\lambda)\geqslant 1$ there is exactly one positive solution. On the other hand,  $2x(\lambda)<1$. Then, we take $\phi\in S^+_1$ such that $\phi(0)=0$ and $\phi$ satisfies 
$$
\phi_x(x)=\left(r(\lambda)^p+\frac{\lambda p}{(p-1)q}(qF(\phi(x))-|\phi(x)|^q)\right)^\frac{1}{p}
$$
for $x\in (0,x(\lambda))$, $\phi(x)=z^+$, for $x\in [x(\lambda),1-x(\lambda)]$
and 
$$
\phi_x(x)=-\left(r(\lambda)^p+\frac{\lambda p}{(p-1)q}(qF(\phi(x))-|\phi(x)|^q)\right)^\frac{1}{p}
$$
for $x\in (1-x(\lambda),1)$. Now, if $q<p$ and $0 <\lambda$ we have exactly one positive solution with flat core when $2x(\lambda)<1$ and with dead core when $2x(\lambda)\geqslant1$. It remains to consider the case $q>p$, in which   we have 
\begin{equation*}
\begin{cases}
\text{at least two positive solution of \eqref{equilibria}},& \text{$\frac{p-1}{p}\left(2I(a_*)\right)^p<\lambda$,}\\
\text{at least one positive solution of \eqref{equilibria}},& \text{$\frac{p-1}{p}\left(2I(a_*)\right)^p=\lambda$.}
\end{cases}
\end{equation*}

\subsection*{Negative solution}

Similarly to the positive case we obtain that $S$ is strictly increasing on $(0,r^-(\lambda))$, from  \eqref{negativezero} and \eqref{alpha} we have
\begin{equation}\label{alphaJ}
\alpha(r)=\left(\frac{p-1}{\lambda p}\right)^{\frac{1}{p}}\int^0_{S(r)}\left(F(t)-F(S(r))+\frac{|S(r)|^q-|t|^q}{q}\right)^{-\frac{1}{p}}dt=\left(\frac{p-1}{\lambda p}\right)^{\frac{1}{p}}J(S(r))
\end{equation}
where
\begin{equation}\label{J}
J(a)=\int^0_{a}\left(F(t)-F(a)+\frac{|a|^q-|t|^q}{q}\right)^{-\frac{1}{p}}dt
\end{equation}
for all $a\in(z^-,0)$.
\begin{lemma}\label{function_J} The function $J$ 
has the following properties    
\begin{enumerate}
\item[i)] $J$ is continuous in $(z^-,0)$
\item[ii)]  \begin{equation*}
\lim_{a\rightarrow 0^-} J(a)=
\begin{cases} 
0,&\ \hbox{ if } q<p,\\
p^{\frac{1}{p}}\int^1_0(1-y^p)^{-\frac{1}{p}}dy,&\ \hbox{ if } q=p,\\
+\infty,&\ \hbox{ if } q>p
\end{cases}
\end{equation*}
\item[iii)]    
\begin{equation*}
	J(z^-)=\begin{cases} 
		+\infty,&\ \hbox{ if } 1<p\leqslant2,\\
		<+\infty,&\ \hbox{ if } p>2.
	\end{cases}
\end{equation*}

\item[iv)]  
\begin{equation*}
\begin{cases} 
\text{$J$ is decreasing in $(z^-,0)$,}&\ \hbox{ if } q\leqslant p,\\
\text{$J(a)\geqslant J(b_*)>0$ for all $a\in(z^-,0)$ for some $b_*\in(z^-,0)$,}&\ \hbox{ if } q>p.
\end{cases}
\end{equation*}
\end{enumerate}
\end{lemma}
That implies
\begin{equation}\label{limitsalpha}
\begin{cases}
\text{$\alpha$ is a increasing function and}\quad\frac{1}{2}\left(\frac{\lambda_1}{\lambda}\right)^\frac{1}{p}\stackrel{r\to 0^+}{\longleftarrow}\alpha(r)\stackrel{r\to r(\lambda)^-}\longrightarrow \alpha(r^-(\lambda)),&\text{if $q=p$},\\
\text{$\alpha$ is a increasing function and}\quad 0\stackrel{r\to 0^+}{\longleftarrow}\alpha(r)\stackrel{r\to r^-(\lambda)^-}\longrightarrow \alpha(r^-(\lambda)),&\text{if $q<p$},\\
\text{$\alpha(r)\geqslant\left(\frac{p-1}{\lambda p}\right)^{\frac{1}{p}}J(b_*)$ for all $r\in(0,r^-(\lambda))$ and}\quad+\infty\stackrel{r\to 0^+}{\longleftarrow}\theta(r)\stackrel{r\to r^-(\lambda)^-}{\longrightarrow} \alpha(r^-(\lambda)),&\text{if $q>p$},
\end{cases}
\end{equation}
where
\begin{equation*}
	\alpha(r^-(\lambda)) = \int_{z^-}^0 \left(r^-(\lambda)^p+\frac{\lambda p}{(p-1)q}(qF(t)- |t|^q)\right)^{-\frac{1}{p}}dt=
	\begin{cases}
		+\infty,& \quad \text{if} \quad1<p\leqslant 2,\\
		y(\lambda)<+\infty,& \quad \text{if} \quad2<p<\infty,
	\end{cases}
\end{equation*} 
since that $S(r)\stackrel{r\to 0^+}{\longrightarrow} 0$  and $S(r)\stackrel{r\to r^-(\lambda)^-}\longrightarrow z^-.$ 

\bigskip

Thus, if $1<p\leqslant 2$  we have 
\begin{equation*}
\begin{cases}
\text{exactly one negative solution of \eqref{equilibria}},& \text{if $\lambda_1<\lambda$ and $q=p$}\\
\text{exactly one negative solution of \eqref{equilibria}},& \text{if $0<\lambda$ and $q<p$}\\
\text{at least two negative solution of \eqref{equilibria}},& \text{if $\lambda^-_{*,1}:=\frac{p-1}{p}\left(2J(b_*)\right)^p<\lambda$ and $q>p$,}\\
\text{at least one negative solution of \eqref{equilibria}},& \text{if $\lambda^-_{*,1}=\lambda$ and $q>p$,}
\end{cases}
\end{equation*}
and for $2<p<+\infty$
\begin{equation*}
\begin{cases}
\text{exactly one negative solution of \eqref{equilibria}},& \text{if $\lambda_1<\lambda$, $2y(\lambda)\geqslant1 $ and $q=p$}\\
\text{exactly one negative solution of \eqref{equilibria}},& \text{if $\lambda_1<\lambda$, $2y(\lambda)<1 $ and $q=p$}\\
\text{exactly one negative solution of \eqref{equilibria}},& \text{if $0<\lambda$, $2y(\lambda)\geqslant1 $ and $q<p$}\\
\text{exactly one negative solution of \eqref{equilibria}},& \text{if $0<\lambda$, $2y(\lambda)<1 $ and $q<p$}\\
\text{at least one negative solution of \eqref{equilibria}},& \text{if $0<\lambda$, $2y(\lambda)<1$ and $q>p$,}\\
\text{at least two negative solution of \eqref{equilibria}},& \text{if $\lambda^-_{*,1}<\lambda$, $1\leqslant 2y(\lambda)$ and $q>p$,}\\
\text{at least one negative solution of \eqref{equilibria}},& \text{if $\lambda^-_{*,1}=\lambda$ and $q>p$.}
\end{cases}
\end{equation*}

\section{Solutions in $S^\pm_{2j}$}\label{S-even}

This section is dedicated to study of  solutions of \eqref{equilibria} that live in $S^\pm_{2j}$. For this we need to study $\frac{1}{j}$ belong to the  image of the function $2(\theta+\alpha):(0,r^*(\lambda))\rightarrow \R^+$, where $r^*(\lambda)=\min\{r(\lambda),r^-(\lambda)\}$. First note that the increasing functions $z:(0,r^*(\lambda))\to (0,z(r^*(\lambda)))\subseteq(0,z^+)$ and $S:(0,r^*(\lambda))\to (S(r^*(\lambda)),0)\subseteq(z^-,0)$ are given implicitly by  \eqref{thetaI} and \eqref{alphaJ} for all $r\in(0,r^*(\lambda))$.

\bigskip

From \eqref{limitstheta} and \eqref{limitsalpha} we have $2\theta+2\alpha$ is increasing for  $q\leqslant p$ and
\begin{equation*}
(2\theta+2\alpha)((0,r^*(\lambda)])=
	\begin{cases}
		\left(2\left(\frac{\lambda_1}{\lambda}\right)^\frac{1}{p},2\theta(r^*(\lambda))+2\alpha(r^*(\lambda))\right],&\text{if $q=p$}\\
		\left(0,2\theta(r^*(\lambda))+2\alpha(r^*(\lambda))\right],&\text{if $q<p$}
	\end{cases}
\end{equation*}
And $2\theta(r)+2\alpha(r)\geqslant \left(\frac{p-1}{\lambda p}\right)^{\frac{1}{p}}2I_e$ for all $ r\in(0,r^*(\lambda))$ if $q>p$ where 
\begin{equation} \label{Ie}
	I_e=\min_{r\in[0,r^*(\lambda)]}\{I(z(r))+J(S(r))\}=I(z(r_e))+J(S(r_e))),\quad\text{for some $r_e\in(0,r^*(\lambda))$ }
\end{equation}
thus
\begin{equation*}
		\begin{cases}
		(2\theta+2\alpha)((0,r_e])=\left[2\left(\frac{p-1}{\lambda p}\right)^\frac{1}{p}I(a_*),+\infty\right],&\text{if $q>p$}\\
		(2\theta+2\alpha)((r_e,r^*(\lambda)])=\left(2\left(\frac{p-1}{\lambda p}\right)^\frac{1}{p}I(a_*),2\theta(r^*(\lambda))+2\alpha(r^*(\lambda))\right],&\text{if $q>p$}
	\end{cases}
\end{equation*}
where
\begin{equation*}
(\theta+\alpha)(r^*(\lambda))  = \int_{S(r^*(\lambda))}^{z(r^*(\lambda))} \left(r^*(\lambda)^p+\frac{\lambda p}{(p-1)q}(qF(t)- |t|^q)\right)^{-\frac{1}{p}}dt=
\begin{cases}
+\infty,& 1<p\leqslant 2,\\
<+\infty,&  2<p.
\end{cases}
\end{equation*} 

\bigskip

First, we suppose  $1<p\leqslant 2$. If $q=p$ and  $\lambda_{2k}=(2k)^p\lambda_1 <\lambda\leqslant \lambda_{2(k+1)}=(2(k+1))^p\lambda_1 $, we have
$$
\frac{1}{j}\in\left(2\left(\frac{\lambda_1}{\lambda}\right)^\frac{1}{p},+\infty\right)=2(\theta+\alpha)((0,r^*(\lambda)))
$$
for $j=1,\cdots,k$, so there is  one $r_j\in (0,r^*(\lambda))$ such that $2j(\theta(r_j)+\alpha(r_j))=1$ and, associated to that,  exactly one $\phi\in S^\pm_{2j}$ solution of \eqref{equilibria}. If $q<p$ and  $0<\lambda$, we have
$$
\frac{1}{j}\in\left(0,+\infty\right)=2(\theta+\alpha)((0,r^*(\lambda)))
$$
for $j=1,2,\cdots,$ thus there is  one $r_j\in (0,r^*(\lambda))$ such that $2j(\theta(r_j)+\alpha(r_j))=1$ and, associated to that,  exactly one $\phi\in S^\pm_{2j}$ solution of \eqref{equilibria}. It remains $q>p$. If  $\lambda^+_{*,2k}=\frac{p-1}{p}\left(2kI_e\right)^p<\lambda\leqslant \frac{p-1}{p}\left((2(k+1))I_e\right)^p=\lambda^+_{*,2(k+1)}$, we have
$$
\frac{1}{j}\in\left[2\left(\frac{p-1}{\lambda p}\right)^\frac{1}{p}I_e,+\infty\right)=2(\theta+\alpha)((0,r^*(\lambda)))
$$
for $j=1,\cdots,k$, so there are at least $r_j,r_j^*\in (0,r^*(\lambda))$ such that $2j(\theta(r_j)+\alpha(r_j))=2j(\theta(r^*_j)+\alpha(r^*_j))=1$ and, associated to those,  two $\phi,\psi\in S_{2j}^\pm$  solution of \eqref{equilibria}. And if $\frac{p-1}{p}\left(2kI_e\right)^p=\lambda$ we have   $2k(\theta(r_e)+\alpha(r_e))=1$ and, associated to that, at least  one $\phi\in S_{2k}^\pm$  solution of \eqref{equilibria}.

\bigskip

Now, we suppose $2<p<+\infty$. If $q=p$ and  $\lambda_{2k}=(2k)^p\lambda_1 <\lambda \leqslant\lambda_{2(k+1)}=(2(k+1))^p\lambda_1 $ for some $k=1,2,\cdots$, we have 
$$
\frac{1}{i}\in\left[ 2\left(\frac{\lambda_1}{\lambda}\right)^\frac{1}{p},2(\theta(r^*(\lambda))+\alpha(r^*(\lambda)))\right)= 2(\theta+\alpha)((0,r^*(\lambda)]),
$$
for $i=j,\cdots,k$ where $j$  satisfies $2j(\theta(r^*(\lambda))+\alpha(r^*(\lambda)))\geqslant 1>2(j-1)(\theta(r^*(\lambda))+\alpha(r^*(\lambda)))$, then for $i=j,\cdots,k$  there is exactly one $\phi\in S^\pm_{2i}$. On the other hand, it is $2i(\theta(r^*(\lambda))+\alpha(r^*(\lambda)))<1$ for $i=1,\cdots,j-1$, we proceed as follows: 
 
\bigskip 

CASE 1; $A(z^+)<A(z^-)$(see the Figure \ref{fig21}). In this case we have $r^*(\lambda)=r(\lambda)<r^-(\lambda)$ since $r(\lambda)^p=\frac{\lambda pA(z^+)}{p-1}$ and $r^-(\lambda)^p=\frac{\lambda pA(z^-)}{p-1}$. Let $(a_1,\cdots,a_{i})\in (\R^+)^{i}$ such that $\displaystyle\sum_{\ell=1}^{i}a_\ell=1-2ix(\lambda)-2i\alpha(r(\lambda))$ and we take $\phi\in S^+_{2i}$ such that $\phi(0)=0$, $\phi(x)=z^+$ for all $x\in I^1_j=\left[x(\lambda)+2(j-1)(x(\lambda)+\alpha(r(\lambda)))+\displaystyle\sum_{\ell=1}^{j-1}a_\ell,x(\lambda)+2(j-1)(x(\lambda)+\alpha(r(\lambda)))+\displaystyle\sum_{\ell=1}^{j}a_\ell\right]$ for $1\leqslant j\leqslant i$ and $\phi$ satisfies 

\begin{equation}\label{dynamical-rlambda}
	|\phi_x(x)|^p=r^*(\lambda)^p+\frac{\lambda p}{(p-1)q}(qF(\phi(x))-|\phi(x)|^q)
\end{equation}
for all $x\in\displaystyle [0,1]\setminus(\cup^i_{j=1}I^1_j)$.

\begin{figure}[h]
	\centering
	
	\begin{tikzpicture}[scale=3.2]
		
		\definecolor{uuuuuu}{rgb}{0.2,0.2,0.2}
		
		\definecolor{ffqqqq}{rgb}{0.9,0.1,0.0}
		
		\draw[<->,color=uuuuuu] (-0.2,0) -- (4.6,0.0);
		\draw[<->,color=uuuuuu] (0.0,-0.8) -- (0.0,0.5);
		
		\draw[-,color=uuuuuu] (0.3,0.0) -- (0.3,0.4);
		\draw[-,color=uuuuuu] (1.0,0.0) -- (1.0,0.4);
		\draw[-,color=uuuuuu] (1.8,0.0) -- (1.8,-0.7);
		\draw[-,color=uuuuuu] (2.6,0.0) -- (2.6,0.4);
		\draw[-,color=uuuuuu] (3.2,0.0) -- (3.2,0.4);
		\draw[-,color=uuuuuu] (4.0,0.0) -- (4.0,-0.7);
		
		\draw[-,color=uuuuuu] (0.0,0.4) -- (0.3,0.4);
		
		\draw[line width=1.2pt,smooth,samples=100,domain=0.0:0.3] plot(\x,{0.4*sin(((\x))*180/(0.6)});
		\draw[line width=1.2pt,smooth,samples=100,domain=0.3:1.0] plot(\x,{0.4});
		\draw[line width=1.2pt,smooth,samples=100,domain=1.0:1.3] plot(\x,{0.4*sin((1.3-(\x))*180/(0.6)});
		\draw[line width=1.2pt,smooth,samples=100,domain=1.3:1.8] plot(\x,{-0.7*sin(((\x)-1.3)*180/(1.0)});
		\draw[line width=1.2pt,smooth,samples=100,domain=1.8:2.3] plot(\x,{-0.7*sin((2.3-(\x))*180/(1.0)});
		\draw[line width=1.2pt,smooth,samples=100,domain=2.3:2.6] plot(\x,{0.4*sin(((\x)-2.3)*180/(0.6)});
		\draw[line width=1.2pt,smooth,samples=100,domain=2.6:3.2] plot(\x,{0.4});
		\draw[line width=1.2pt,smooth,samples=100,domain=3.2:3.5] plot(\x,{0.4*sin((3.5-(\x))*180/(0.6)});
		\draw[line width=1.2pt,smooth,samples=100,domain=3.5:4.0] plot(\x,{-0.7*sin(((\x)-3.5)*180/(1.0)});
		\draw[line width=1.2pt,smooth,samples=100,domain=4.0:4.5] plot(\x,{-0.7*sin((4.5-(\x))*180/(1.0)});

		\begin{scriptsize}
			
			\draw[color=uuuuuu] (1.3,-0.25) node {$\phi$};
			
			\draw [fill=uuuuuu] (0.0,0.4) circle (0.4pt);
			\draw[color=uuuuuu] (-0.1,0.4) node {$z^+$};
			
			\draw[color=uuuuuu] (0.15,0.05) node {$x(\lambda)$};
			\draw[color=uuuuuu] (1.15,0.05) node {$x(\lambda)$};
			\draw[color=uuuuuu] (0.65,0.05) node {$a_1$};
			\draw[color=uuuuuu] (1.55,0.05) node {$\alpha(r(\lambda))$};
			\draw[color=uuuuuu] (2.05,0.05) node {$\alpha(r(\lambda))$};
			\draw[color=uuuuuu] (2.90,0.05) node {$a_2$};
			\draw[color=uuuuuu] (2.45,0.05) node {$x(\lambda)$};
			\draw[color=uuuuuu] (3.35,0.05) node {$x(\lambda)$};
			\draw[color=uuuuuu] (3.75,0.05) node {$\alpha(r(\lambda))$};
			\draw[color=uuuuuu] (4.25,0.05) node {$\alpha(r(\lambda))$};

			\draw [fill=uuuuuu] (0.0,0.0) circle (1.0pt);
			\draw [fill=uuuuuu] (1.3,0.0) circle (1.0pt);
			\draw [fill=uuuuuu] (2.3,0.0) circle (1.0pt);
			\draw [fill=uuuuuu] (3.5,0.0) circle (1.0pt);
			\draw [fill=uuuuuu] (4.5,0.0) circle (1.0pt);
			\draw[color=uuuuuu] (4.56,-0.06) node {$1$};
			\draw[color=uuuuuu] (-0.06,-0.06) node {$0$};
			
		\end{scriptsize}
	\end{tikzpicture}
	\caption{Case 1, $A(z^+)<A(z^-)$ : $\phi\in S^+_4$} \label{fig21}	
\end{figure}	

\bigskip

CASE 2; $A(z^+)>A(z^-)$  (to see the Figure \ref{fig22}).  Let $(a_1,\cdots,a_{i})\in (\R^+)^{i}$ such that $\displaystyle\sum_{\ell=1}^{i}a_\ell=1-2iy(\lambda)-2i\theta(r^-(\lambda))$ and we take $\phi\in S^+_{2i}$ such that $\phi(0)=0$, $\phi(x)=z^-$ for $x\in I^2_j=\left[y(\lambda)+2(j-1)(y(\lambda)+\theta(r^-(\lambda))))+\displaystyle\sum_{\ell=1}^{j-1}a_\ell,y(\lambda)+2(j-1)(y(\lambda)+\theta(r^-(\lambda)))+\displaystyle\sum_{\ell=1}^{j}a_\ell\right]$ for $1\leqslant j\leqslant i$ and $\phi$ satisfies \eqref{dynamical-rlambda}
for all $x\in\displaystyle [0,1]\setminus(\cup^i_{j=1}I^2_j)$.

%
%
%
\begin{figure}[h]
	\centering
	
	\begin{tikzpicture}[scale=3.2]
		
		\definecolor{uuuuuu}{rgb}{0.2,0.2,0.2}
		
		\definecolor{ffqqqq}{rgb}{0.9,0.1,0.0}
		
		\draw[<->,color=uuuuuu] (-0.2,0) -- (4.6,0.0);
		\draw[<->,color=uuuuuu] (0.0,-0.8) -- (0.0,0.5);
		
		\draw[-,color=uuuuuu] (0.3,0.0) -- (0.3,0.4);
		\draw[-,color=uuuuuu] (1.1,0.0) -- (1.1,-0.7);
		\draw[-,color=uuuuuu] (1.4,0.0) -- (1.4,-0.7);
		\draw[-,color=uuuuuu] (2.2,0.0) -- (2.2,0.4);
		\draw[-,color=uuuuuu] (3.0,0.0) -- (3.0,-0.7);
		\draw[-,color=uuuuuu] (4.0,0.0) -- (4.0,-0.7);
		
		\draw[-,color=uuuuuu] (0.0,-0.7) -- (1.1,-0.7);
		
		\draw[line width=1.2pt,smooth,samples=100,domain=0.0:0.3] plot(\x,{0.4*sin(((\x))*180/(0.6)});
		\draw[line width=1.2pt,smooth,samples=100,domain=0.3:0.6] plot(\x,{0.4*sin((0.6-(\x))*180/(0.6)});
		\draw[line width=1.2pt,smooth,samples=100,domain=0.6:1.1] plot(\x,{-0.7*sin(((\x)-0.6)*180/(1.0)});
		\draw[line width=1.2pt,smooth,samples=100,domain=1.1:1.4] plot(\x,{-0.7});
		\draw[line width=1.2pt,smooth,samples=100,domain=1.4:1.9] plot(\x,{-0.7*sin((1.9-(\x))*180/(1.0)});
		\draw[line width=1.2pt,smooth,samples=100,domain=1.9:2.2] plot(\x,{0.4*sin(((\x)-1.9)*180/(0.6)});
		\draw[line width=1.2pt,smooth,samples=100,domain=2.2:2.5] plot(\x,{0.4*sin((2.5-(\x))*180/(0.6)});
		\draw[line width=1.2pt,smooth,samples=100,domain=2.5:3.0] plot(\x,{-0.7*sin(((\x)-2.5)*180/(1.0)});
		\draw[line width=1.2pt,smooth,samples=100,domain=3.0:4.0] plot(\x,{-0.7});
		\draw[line width=1.2pt,smooth,samples=100,domain=4.0:4.5] plot(\x,{-0.7*sin((4.5-(\x))*180/(1.0)});

		\begin{scriptsize}
			
			\draw[color=uuuuuu] (1.9,-0.25) node {$\phi$};
			
			\draw [fill=uuuuuu] (0.0,-0.7) circle (0.4pt);
			\draw[color=uuuuuu] (-0.1,-0.7) node {$z^-$};
			
			\draw[color=uuuuuu] (0.3,-0.05) node {$2\theta(r^-(\lambda))$};
			\draw[color=uuuuuu] (0.85,0.05) node {$y(\lambda)$};
			\draw[color=uuuuuu] (1.65,0.05) node {$y(\lambda)$};
			\draw[color=uuuuuu] (1.25,0.05) node {$a_1$};
			\draw[color=uuuuuu] (2.2,-0.05) node {$2\theta(r^-(\lambda))$};
			\draw[color=uuuuuu] (3.50,0.05) node {$a_2$};
			\draw[color=uuuuuu] (2.75,0.05) node {$y(\lambda)$};
			\draw[color=uuuuuu] (4.25,0.05) node {$y(\lambda)$};

			\draw [fill=uuuuuu] (0.0,0.0) circle (1.0pt);
			\draw [fill=uuuuuu] (0.6,0.0) circle (1.0pt);
			\draw [fill=uuuuuu] (1.9,0.0) circle (1.0pt);
			\draw [fill=uuuuuu] (2.5,0.0) circle (1.0pt);
			\draw [fill=uuuuuu] (4.5,0.0) circle (1.0pt);
			\draw[color=uuuuuu] (4.56,-0.06) node {$1$};
			\draw[color=uuuuuu] (-0.06,-0.06) node {$0$};
			
		\end{scriptsize}
	\end{tikzpicture}
	\caption{Case 2, $A(z^+)>A(z^-)$: $\phi\in S^+_4$} \label{fig22}	
\end{figure}	

\bigskip

CASE 3; $A(z^+)=A(z^-)$ (to see the Figure \ref{fig23}). In this case we have  $r^*(\lambda)=r(\lambda)=r^-(\lambda)$. Let $(a_1,\cdots,a_{i})\in (\R^+)^{i}$ such that $\displaystyle\sum_{\ell=1}^{i}a_\ell=1-2ix(\lambda)-2iy(r(\lambda))$ and we take $\phi\in S^+_{2i}$ such that $\phi(0)=0$, $\phi(x)=z^+$ or $\phi(x)=z^-$ alternately for all $x\in I^3_j=\left[x(\lambda)+2(j-1)(x(\lambda)+y(\lambda))+\displaystyle\sum_{\ell=1}^{j-1}a_\ell,x(\lambda)+2(j-1)(x(\lambda)+y(\lambda))+\displaystyle\sum_{\ell=1}^{j}a_\ell\right]$ for $1\leqslant j\leqslant i$ and $\phi$ satisfies 
for all $x\in\displaystyle [0,1]\setminus(\cup^i_{j=1}I^3_j)$.

\begin{figure}[h]
	\centering
	
	\begin{tikzpicture}[scale=3.2]
	
	\definecolor{uuuuuu}{rgb}{0.2,0.2,0.2}
	
	\definecolor{ffqqqq}{rgb}{0.9,0.1,0.0}
	
	\draw[<->,color=uuuuuu] (-0.2,0) -- (4.6,0.0);
	\draw[<->,color=uuuuuu] (0.0,-0.8) -- (0.0,0.5);
	
	\draw[-,color=uuuuuu] (0.3,0.0) -- (0.3,0.4);
	\draw[-,color=uuuuuu] (0.7,0.0) -- (0.7,0.4);
	\draw[-,color=uuuuuu] (1.5,0.0) -- (1.5,-0.7);
	\draw[-,color=uuuuuu] (1.7,0.0) -- (1.7,-0.7);
	\draw[-,color=uuuuuu] (2.5,0.0) -- (2.5,0.4);
	\draw[-,color=uuuuuu] (2.8,0.0) -- (2.8,0.4);
	\draw[-,color=uuuuuu] (3.6,0.0) -- (3.6,-0.7);
	\draw[-,color=uuuuuu] (4.0,0.0) -- (4.0,-0.7);
	
	\draw[-,color=uuuuuu] (0.0,0.4) -- (0.3,0.4);
	\draw[-,color=uuuuuu] (0.0,-0.7) -- (1.5,-0.7);
	
	\draw[line width=1.2pt,smooth,samples=100,domain=0.0:0.3] plot(\x,{0.4*sin(((\x))*180/(0.6)});
	\draw[line width=1.2pt,smooth,samples=100,domain=0.3:0.7] plot(\x,{0.4});
	\draw[line width=1.2pt,smooth,samples=100,domain=0.7:1.0] plot(\x,{0.4*sin((1.0-(\x))*180/(0.6)});
	\draw[line width=1.2pt,smooth,samples=100,domain=1.0:1.5] plot(\x,{-0.7*sin(((\x)-1.0)*180/(1.0)});
	\draw[line width=1.2pt,smooth,samples=100,domain=1.5:1.7] plot(\x,{-0.7});
	\draw[line width=1.2pt,smooth,samples=100,domain=1.7:2.2] plot(\x,{-0.7*sin((2.2-(\x))*180/(1.0)});
	\draw[line width=1.2pt,smooth,samples=100,domain=2.2:2.5] plot(\x,{0.4*sin(((\x)-2.2)*180/(0.6)});
	\draw[line width=1.2pt,smooth,samples=100,domain=2.5:2.8] plot(\x,{0.4});
	\draw[line width=1.2pt,smooth,samples=100,domain=2.8:3.1] plot(\x,{0.4*sin((3.1-(\x))*180/(0.6)});
		\draw[line width=1.2pt,smooth,samples=100,domain=3.1:3.6] plot(\x,{-0.7*sin(((\x)-3.1)*180/(1.0)});
	\draw[line width=1.2pt,smooth,samples=100,domain=3.6:4.0] plot(\x,{-0.7});
	\draw[line width=1.2pt,smooth,samples=100,domain=4.0:4.5] plot(\x,{-0.7*sin((4.5-(\x))*180/(1.0)});

	\begin{scriptsize}
	
	\draw[color=uuuuuu] (2.2,-0.25) node {$\phi$};
	
	\draw [fill=uuuuuu] (0.0,-0.7) circle (0.4pt);
	\draw[color=uuuuuu] (-0.1,-0.7) node {$z^-$};
	\draw [fill=uuuuuu] (0.0,0.4) circle (0.4pt);
	\draw[color=uuuuuu] (-0.1,0.4) node {$z^+$};

	\draw[color=uuuuuu] (0.15,0.05) node {$x(\lambda)$};
	\draw[color=uuuuuu] (0.85,0.05) node {$x(\lambda)$};
	\draw[color=uuuuuu] (0.5,0.05) node {$a_1$};
				\draw[color=uuuuuu] (1.25,0.05) node {$y(\lambda)$};
	\draw[color=uuuuuu] (1.95,0.05) node {$y(\lambda)$};
	\draw[color=uuuuuu] (1.6,0.05) node {$a_2$};
		\draw[color=uuuuuu] (2.35,0.05) node {$x(\lambda)$};
	\draw[color=uuuuuu] (2.95,0.05) node {$x(\lambda)$};
	\draw[color=uuuuuu] (2.65,0.05) node {$a_3$};
		\draw[color=uuuuuu] (3.80,0.05) node {$a_4$};
	\draw[color=uuuuuu] (3.35,0.05) node {$y(\lambda)$};
	\draw[color=uuuuuu] (4.25,0.05) node {$y(\lambda)$};

	\draw [fill=uuuuuu] (0.0,0.0) circle (1.0pt);
	\draw [fill=uuuuuu] (1.0,0.0) circle (1.0pt);
	\draw [fill=uuuuuu] (2.2,0.0) circle (1.0pt);
	\draw [fill=uuuuuu] (3.1,0.0) circle (1.0pt);
	\draw [fill=uuuuuu] (4.5,0.0) circle (1.0pt);
	\draw[color=uuuuuu] (4.56,-0.06) node {$1$};
	\draw[color=uuuuuu] (-0.06,-0.06) node {$0$};
	
	\end{scriptsize}
	\end{tikzpicture}
	\caption{Case 3, $A(z^+)=A(z^-)$: $\phi\in S^+_4$} \label{fig23}	
\end{figure}

\bigskip

 Now, if $q<p$ and  $0<\lambda$, we have 
$$
\frac{1}{i}\in\left(0,2(\theta(r^*(\lambda))+\alpha(r^*(\lambda)))\right]= 2(\theta+\alpha)((0,r^*(\lambda)]),
$$
for $i=j,j+1\cdots,\infty$ for some $j$ that satisfies $2j(\theta(r^*(\lambda))+\alpha(r^*(\lambda)))\geqslant 1>2(j-1)(\theta(r^*(\lambda))+\alpha(r^*(\lambda)))$, then for $i=j,j+1, \cdots$  there is exactly one $\phi\in S^\pm_{2i}$. On the other hand, it is $2i(\theta(r^*(\lambda))+\alpha(r^*(\lambda)))<1$ for $i=1,\cdots j-1$, we proceed  as in the case $q=p>2$. It remains $q>p$. If $\frac{p-1}{p}\left(2kI_e\right)^p\leqslant\lambda\leqslant \frac{p-1}{p}\left((2(k+1))I_e\right)^p$, we have 
$$
\frac{1}{i}\in\left[ 2\left(\frac{p-1}{\lambda p}\right)^\frac{1}{p}I(a_*),+\infty\right)= 2(\theta+\alpha)((0,r_e]),
$$
for $j=1,\cdots,k$, so there is at least one $r_j\in(0,r_e]$ such that $2j(\theta(r_j)+\alpha(r_j))=1$, and associated to those one $\phi\in^\pm_{2j}$ solution of \eqref{equilibria} with dead core.
And if $\frac{p-1}{p}\left(2kI_e\right)^p<\lambda\leqslant \frac{p-1}{p}\left((2(k+1))I_e\right)^p$
$$
\frac{1}{i}\in\left( 2\left(\frac{p-1}{\lambda p}\right)^\frac{1}{p}I(a_*),2(\theta(r^*(\lambda))+\alpha(r^*(\lambda)))\right)= 2(\theta+\alpha)((r_e,r^*(\lambda)]),
$$
for $i=j,\cdots,k$ for some $j$ that satisfies $2j(\theta(r^*(\lambda))+\alpha(r^*(\lambda)))\geqslant 1>2(j-1)(\theta(r^*(\lambda))+\alpha(r^*(\lambda)))$, then for $i=j,\cdots,k$ there are at least  one $\phi,\psi\in S^\pm_{2i}$  solutions of \eqref{equilibria} with dead core.  On the other hands, it is $2i(\theta(r^*(\lambda))+\alpha(r^*(\lambda)))<1$ for $i=1,\cdots, j-1$, we construct the solutions with flat core as in the case $q=p>2$.

\section{Solutions in $S^\pm_{2j-1}$}\label{S-odd}

This section is dedicated to study the existence of solutions of \eqref{equilibria} that live in $S^+_{2j-1}$ for some $j=2,3,\cdots$ (for solutions that live in $S^-_{2j-1}$ is done similarly by changing the roles of  $\theta$ and $\alpha$). In fact, we need to study that $\frac{1}{j}$ belong to the image of  function $\frac{2\theta+2\alpha}{1+2\alpha}:(0,r^*(\lambda))\rightarrow \R^+$, where $r^*(\lambda)=\min\{r(\lambda),r^-(\lambda)\}$.  

\bigskip

From \eqref{limitstheta} and \eqref{limitsalpha}
\begin{equation*}
	\begin{cases}
	2-\frac{2}{\left(\frac{\lambda_1}{\lambda}\right)^\frac{1}{p}+1}\stackrel{r\to 0^+}{\longleftarrow}\frac{2\theta(r)+2\alpha(r)}{1+2\alpha(r)}\stackrel{r\to r^*(\lambda)^-}\longrightarrow \frac{2\theta(r^*(\lambda))+2\alpha(r^*(\lambda))}{1+2\alpha(r^*(\lambda))},&\text{if $q=p$,}\\0\stackrel{r\to 0^+}{\longleftarrow}\frac{2\theta(r)+2\alpha(r)}{1+2\alpha(r)}\stackrel{r\to r^*(\lambda)^-}\longrightarrow \frac{2\theta(r^*(\lambda))+2\alpha(r^*(\lambda))}{1+2\alpha(r^*(\lambda))},&\text{if $q<p$,}\\2\stackrel{r\to 0^+}{\longleftarrow}\frac{2\theta(r)+2\alpha(r)}{1+2\alpha(r)}\stackrel{r\to r^*(\lambda)^-}\longrightarrow \frac{2\theta(r^*(\lambda))+2\alpha(r^*(\lambda))}{1+2\alpha(r^*(\lambda))},&\text{if $q>p$.}
	\end{cases}
\end{equation*}

\bigskip

First, we suppose  $1<p\leqslant 2$, we have
\begin{equation*}
\begin{cases}
\left(2-\frac{2}{\left(\frac{\lambda_1}{\lambda}\right)^\frac{1}{p}+1},1\right)\subseteq (\frac{2\theta+2\alpha}{1+2\alpha})(0,r^*(\lambda)),& \text{if $q=p$,}\\
\left(0,1\right)\subseteq (\frac{2\theta+2\alpha}{1+2\alpha})(0,r^*(\lambda)),& \text{if $q<p$,}\\
\end{cases}
\end{equation*} 
since $\frac{2\theta(r^*(\lambda))+2\alpha(r^*(\lambda))}{1+2\alpha(r^*(\lambda))}\geqslant 1$. If $q=p$ and  $\lambda_{2k-1}=(2k-1)^p\lambda_1 <\lambda \leqslant\lambda_{2k+1}=(2k+1)^p\lambda_1 $ for some $k=2,3,\cdots$, we have $\frac{1}{k}>2-\frac{2}{\left(\frac{\lambda_1}{\lambda}\right)^\frac{1}{p}+1}> \frac{1}{k+1}$. Thus for $j=2,\cdots,k$  there is  at least one $\phi\in S^+_{2j-1}$ solution of \eqref{equilibria}. If $q<p$ and $0<\lambda$, we have, for $j=2,3,\cdots,$  there is at least one  $\phi\in S^+_{2j-1}$ solution of \eqref{equilibria}. 

\bigskip

Now, we suppose $2<p<+\infty$. If $q=p$ and  $\lambda_{2k-1}=(2k-1)^p\lambda_1 <\lambda \leqslant\lambda_{2k+1}=(2k+1)^p\lambda_1 $ for some $k=2,3,\cdots$, we have for $i=j,\cdots,k$ for some $j$ that satisfies $2j\theta(r^*(\lambda))+2(j-1)\alpha(r^*(\lambda)))\geqslant 1>2(j-1)\theta(r^*(\lambda))+2(j-2)\alpha(r^*(\lambda)))$ one $\phi\in S^+_{2i-1}$ solution of \eqref{equilibria} with dead core, on the other hand, it is for $i=1,\cdots j-1$, we have one $\phi\in S^+_{2i-1}$ solution of \eqref{equilibria} with flat core. If $q<p$ and  $0<\lambda$,  for $i=j,j+1\cdots,$ for some $j$ that satisfies $2j\theta(r^*(\lambda))+2(j-1)\alpha(r^*(\lambda)))\geqslant 1>2(j-1)\theta(r^*(\lambda))+2(j-2)\alpha(r^*(\lambda)))$ we have one $\phi\in S^+_{2i-1}$ solution of \eqref{equilibria} with dead core, on the other hand, it is for $i=1,\cdots j-1$, we have one $\phi\in S^+_{2i-1}$ solution of \eqref{equilibria} with flat core.

\bigskip

It remains to consider $q>p$, in this case we have
\begin{equation*}
	\begin{cases}
		\frac{2\theta+2\alpha}{1+2\alpha}((0,r^+_o])=\left[\frac{1}{\left(\frac{\lambda p}{p-1}\right)^{1/p}\frac{1}{2(I_o^++J^+_o)}+\frac{J^+_o}{I^+_o+J^+_o}},2\right),&\text{if $q>p$}\\
		\frac{2\theta+2\alpha}{1+2\alpha}((r^+_o,r^*(\lambda)])=\left(\frac{1}{\left(\frac{\lambda p}{p-1}\right)^{1/p}\frac{1}{2(I_o^++J^+_o)}+\frac{J^+_o}{I^+_o+J^+_o}},\frac{2\theta(r^*(\lambda))+2\alpha(r^*(\lambda))}{1+2\alpha(r^*(\lambda))}\right],&\text{if $q>p$}
	\end{cases}
\end{equation*}
 where 
\begin{equation} \label{IJpositiveodd}
\frac{1}{\left(\frac{\lambda p}{p-1}\right)^{1/p}\frac{1}{2(I_o^++J^+_o)}+\frac{J^+_o}{I^+_o+J^+_o}}:=\displaystyle\min_{r\in(0,r^*(\lambda))}\left\{\frac{2\theta(r)+2\alpha(r)}{1+2\alpha(r)}\right\}
\end{equation}
with $I^+_o=I(z(r^+_0))$ and $J^+_o=J(S(r^+_o)))$ for some $r^+_o\in(0,r^*(\lambda)).$

\bigskip 

Thus for $1<p\leq2$, we have $\frac{2\theta(r^*(\lambda))+2\alpha(r^*(\lambda))}{1+2\alpha(r^*(\lambda))}\geqslant1$
that can be infinity in any cases, in otherwise, it is $2<p$,  always $\frac{2\theta(r^*(\lambda))+2\alpha(r^*(\lambda))}{1+2\alpha(r^*(\lambda))}$ is finite number. Then, if $\lambda^+_{*,2k-1}=\frac{p-1}{p}\left((k-\frac{J^+_o}{I^+_o+J^+_o})2(I^+_o+J^+_o)\right)^p<\lambda< \frac{p-1}{p}\left(((k+1)-\frac{J^+_o}{I^+_o+J^+_o})2(I^+_o+J^+_o)\right)^p=\lambda^+_{*,2(k+1)-1}$ for some $k=2,\cdots$, we have
$$
\displaystyle\frac{1}{k}> \frac{1}{\left(\frac{\lambda p}{p-1}\right)^{1/p}\frac{1}{2(I_o^++J^+_o)}+\frac{J^+_o}{I^+_o+J^+_o}}>\frac{1}{k+1}
$$
and
$$
\frac{1}{j}\in\left[ \frac{1}{\left(\frac{\lambda p}{p-1}\right)^{1/p}\frac{1}{2(I_o^++J^+_o)}+\frac{J^+_o}{I^+_o+J^+_o}},2\right)\subseteq\frac{2\theta+2\alpha}{1+2\alpha}((0,r^*(\lambda)))
$$
for $j=2,\cdots,k$, that implies there are at least $r_j,r_j^*\in (0,r^*(\lambda))$ such that $2j\theta(r_j)+2(j-1)\alpha(r_j))=2j\theta(r^*_j)+2(j-1)\alpha(r^*_j)=1$ and, associated to those,  two $\phi,\psi\in S_{2j-1}^+$  solution of \eqref{equilibria}. And if $\frac{p-1}{p}\left((k-\frac{J^+_o}{I^+_o+J^+_o})2(I^+_o+J^+_o)\right)^p=\lambda$ we have  $2k\theta(r^+_o)+2(k-1)\alpha(r^+_o))=1$ and, associated to that,  one $\phi\in S_{2k-1}^+$  solution of \eqref{equilibria}. On the other hand, the case $2<p$,  we proceed as above.

\section{Regularity}\label{S-reg} 

In this Section, we study the regularity properties of the  solutions of the following  nonlinear eigenvalue problem
\begin{equation}\label{equilibria-regularity}
	\begin{cases}
		-\left(|\psi_x|^{p-2}\psi_x\right)_x=h(\psi), &\text{in $(0,1)$},\\ 
		\psi(0)=\psi(1)=0&
	\end{cases}	
\end{equation}
where $h$ is a continuous function that satisfies the conditions for existence of solution of \eqref{equilibria-regularity}.  We state a regularity result for the solutions of \eqref{equilibria-regularity}. This result, extend the regularity result of \cite{otani}
\begin{theorem}\label{regularityelliptic}
	Let $\psi$ a solution of \eqref{equilibria-regularity} and $h$ continuous function. We have	
	\begin{equation*}
		\begin{cases}
			\psi\in C^{2}([0,1]), &\text{If $1<p\leqslant2$}\\
			\psi\in C^{1,\frac{1}{p-1}}([0,1])\cap C^2\left([0,1]\setminus(\mathcal{C}\setminus\mathcal{Z})\right),&\text{If $2<p<2(n+1)$}\\
			\psi\in C^{1,\frac{1}{p-1}}([0,1])\cap C^2\left([0,1]\setminus\mathcal{C}\right),&\text{If $2(n+1)\leqslant p$}
					\end{cases}
	\end{equation*}
	where  $\mathcal{Z}=\{\chi\in[0,1]; \text{$\psi(\chi)$ is zero of order $n$ of $h$}\}$ and $\mathcal{C}=\{\chi\in[0,1];\psi_x(\chi)=0\}$. Moreover 
	\begin{equation*}
		\begin{cases}
			\displaystyle\lim_{x\to\chi}\frac{\psi_x(x)}{|x-\chi|^{\frac{1}{p-1}}}=\left(\frac{1}{p-1}	\right)^{\frac{1}{p-1}}\left|h(\psi(\chi))\right|^{\frac{1}{p-1}},		&\text{if $\chi\in\mathcal{C}$}\\
			\displaystyle\lim_{x\to\chi}\frac{\psi_x(x)}{\left| x-\chi\right|^{\theta}}=0.&\text{for all  $\frac{1}{p-1}\leqslant\theta\leqslant1$ if $\chi\in\mathcal{C}\cap\mathcal{Z}$.}
		\end{cases}
	\end{equation*}
\end{theorem}

\begin{proof}
The proof will be done for the positive solution of \eqref{equilibria-regularity}, although the proof is valid for  other solutions. Let $\psi$ the positive  solution  of the elliptic problem \eqref{equilibria-regularity}. First,  we have $\psi\in C^1([0,1])$, since $h(\psi)\in C([0,1])$. Remain, show more regularity. In fact, we multiplying by $x\in\mathcal{C}$ to the equation \eqref{equilibria-regularity} and integrating from to $0$ to $0<x$  we have 
$$
|\psi_x(x)|^p=|\psi_x(0)|^p-\frac{p}{p-1}H(\psi(x))
$$
where $H(\mu)=\displaystyle\int^\mu_0h(\tau)d\tau$ that implies that 
$$
\psi_x(x)=\left[|\psi_x(0)|^p-\frac{p}{p-1}H(\psi(x))\right]^{1/p}
$$
for $0<x\leqslant \chi$.

\bigskip

So differentiating
\begin{equation}\label{secondderivate}
			\psi_{xx}(x)=-\frac{1}{p-1}h(\psi(x))|\psi_x(x)|^{2-p}.
\end{equation}
for all $x\not\in\mathcal{C}$. Thus, $\psi\in C^2([0,1])$ since $1<p\leqslant2$ and $\psi\in C^2([0,1]\setminus \mathcal{C})$  for otherwise $2<p$. Now, let $\chi\in \mathcal{C}$, then we have
\begin{equation}\label{derivatezero}
\psi_x(x)=\left[\frac{p}{p-1}H(\psi(\chi))-\frac{p}{p-1}H(\psi(x))\right]^{1/p}
\end{equation}
for $0<x\leqslant \chi$ and
$$
x=\int^x_0\left[\frac{p}{p-1}H(\psi(\chi))-\frac{ p}{p-1}H(\psi(y))\right]^{-1/p}\psi_y(y)dy
$$
since $\psi_x(\chi)=0$, $\psi_x(x),\psi(x)>0$ for $0<x<\chi$. Thus, by L'hospital 
\begin{equation*}
	\begin{split}
		\displaystyle\lim_{x\to\chi} \frac{(\psi_x(x))^{p-1}}{x-\chi}=&\lim_{x\to\chi} \frac{\left[\frac{ p}{p-1}H(\psi(\chi))-\frac{ p}{p-1}H(\psi(x))\right]^{\frac{p-1}{p}}}{\int^x_{\chi}\left[\frac{p}{p-1}H(\psi(\chi))-\frac{ p}{p-1}H(\psi(y))\right]^{-1/p}\psi_y(y)dy}\\
		=&\lim_{x\to\chi}-\frac{1}{p-1}h(\psi(x))\\
		=&-\frac{1}{p-1}h(\psi(\chi)),
	\end{split}
\end{equation*}
for all $\chi\in \mathcal{C}.$ Thus $\psi\in C^{1,\frac{1}{p-1}}([0,1])$.
		
\bigskip
		
Finally, in some $\chi\in\mathcal{C}\cap\mathcal{Z}$ we have more regularity, exactly  $\psi_{xx}(\chi)=0 $. In fact, we noted that $h(\psi(\chi))=0=(|\psi_x|^{p-2}\psi_x)|_{x=\chi}$. Then from \eqref{secondderivate} and \eqref{derivatezero}
\begin{equation*}
\begin{split}
\lim_{x\to\chi}\psi_{xx}(x)=&\lim_{x\to\chi}\frac{-h(\psi(x))\left[\frac{ p}{p-1}H(\psi(\chi))-\frac{ p}{p-1}H(\psi(x))	\right]^{\frac{2}{p}}}{ p\left[H(\psi(\chi))-H(\psi(x))	\right]}
\end{split}
\end{equation*}

\bigskip

Now we need analyze the last limit. As $\chi\in \mathcal{Z}$, we can write $H(\psi(\chi))-H(t)=K(t)(\psi(\chi)-t)^{n+1}$ for some $K\in C^1([0,\psi(\chi)])$ such that $K(t)>0$ for all $t\in[0,\psi(\chi)]$. Thus 
\begin{equation*}
	\begin{split}
		\lim_{t\to \psi(\chi)}\frac{-h(t)}{\left(H(\psi(\chi))-H(t)\right)^{\frac{p-2}{p}}}=&\lim_{t\to \psi(\chi)}\frac{K'(t)(\psi(\chi)-t)^{n+1}-K(t)(n+1)(\psi(\chi)-t)^{n}}{\left(K(t)\right)^{\frac{p-2}{p}}(\psi(\chi)-t)^{\frac{(p-2)(n+1)}{p}}}\\
		=&\lim_{t\to \psi(\chi)}\frac{K'(t)(\psi(\chi)-t)-K(t)(n+1)}{\left(K(t)\right)^{\frac{p-2}{p}}}(\psi(\chi)-t)^{\frac{2(n+1)-p}{p}}\\
		=&\begin{cases}
			0,&2<p<2(n+1)\\
			-(n+1)(K(\psi(\chi)))^{\frac{1}{n+1}},&p=2(n+1)\\
			+\infty,& 2(n+1)<p.
		\end{cases}
	\end{split}
\end{equation*}
So, using the mean value theorem we conclude that $\psi_{xx}(\chi)=0$ for $\chi\in\mathcal{C}\cap\mathcal{Z}$ and $2<p<2(n+1)$. 
\end{proof}

\medskip

Received xxxx 20xx; revised xxxx 20xx.

\medskip


\begin{thebibliography}{99}
	
%
%
%
%
%
%
%
\bibitem{Bcv}(MR4030598)
\newblock R. C. D. S. Broche, A.N. Carvalho  and  J. Valero.
\newblock A non-autonomous scalar one-dimensional dissipative parabolic problem: the description of the dynamics,
\newblock \emph{Nonlinearity},  \textbf{32} (2019), 4912--4941.
\newblock DOI: 10.1088/1361-6544/ab3f55
%
%
%
\bibitem{Carvalho et al.} (MR2976449)
\newblock A.N. Carvalho, J. A. Langa and J. C. Robinson,
\newblock \emph{Attractors for infinite-dimensional non-autonomous dynamical systems},
\newblock Springer: New York, 2013.
\newblock DOI: 10.1088/1361-6544/ab3f55
%
%
\bibitem{CLR-PAMS} (MR2898698)
\newblock A. N. Carvalho, J.A. Langa and J. C. Robinson,
\newblock {Structure and bifurcation of pullback attractors in a non-autonomous Chafee-Infante equation},
\newblock \emph{Proc. Amer. Math. Soc.}, \textbf{140} (2012), 2357--2373.
\newblock DOI:10.1090/S0002-9939-2011-11071-2
%
%
	\bibitem{CH-IN} (MR440205)
	\newblock N. Chafee and E. F. Infante,
	\newblock A bifurcation problem for a nonlinear partial differential equation of parabolic type,
	\newblock \emph{Applicable Anal.}, \textbf{4}  (1974), {17--37}.
	\newblock DOI: 10.1080/00036817408839081
	
	\bibitem{Chaf-Inf} (MR328359)
	\newblock N. Chafee and E. F. Infante,
	\newblock Bifurcation and stability for a nonlinear parabolic partial differential equation,
	\newblock \emph{Bull. Amer. Math. Soc.}, \textbf{80} (1974), 49--52.
	\newblock DOI: 10.1090/S0002-9904-1974-13349-5
%
%
%
%

\bibitem{Ha}(MR0941371)
\newblock J. K. Hale,
\newblock Asymptotic Behavior of Dissipative Systems,
\newblock \emph{Providence: Math. Surveys and Monographs, American Mathematical Society}, (1998).

\bibitem{HE} (MR610244)
\newblock D. Henry,
\newblock \emph{Geometric theory of semilinear parabolic equations},
\newblock  Springer, Berlin, 1981.
%
%
	\bibitem{Guedda-Veron} (MR0965762)
	\newblock M. Guedda and L. Veron,
	\newblock Bifurcation phenomena associated to the p-Laplace operator,
	\newblock \emph{Trans. Amer. Math. Soc.}, \textbf{310} (1988), 419--431.
	\newblock DOI: 10.2307/2001132
	
\bibitem{LICALUMO}(MR4160009)
\newblock Y. Li, A. N. Carvalho, T. L. M. Luna  and  E. M. Moreira,
\newblock A non-autonomous bifurcation problem for a non-local scalar one-dimensional parabolic equation,
\newblock \emph{Commun. Pure Appl. Anal},  \textbf{19} (2020) no. 11, 5181--5196.
\newblock DOI: 10.3934/cpaa.2020232

	\bibitem{otani}(MR0764911)
	\newblock M. Ôtani,
	\newblock On certain second order ordinary differential equations associated with Sobolev-Poincare-type inequalities,
	\newblock \emph{Nonlinear Analysis},  \textbf{8} (1984), 1255--1270.
	\newblock DOI: 10.1016/0362-546X(84)90014-2
	
	
	\bibitem{TAYA} (MR1769251)
	\newblock S. Takeuchi, Y. Yamada,
	\newblock  Asymptotic properties of a reaction-diffusion equation with degenerate p-Laplacian,
	\newblock \emph{Nonlinear Analysis}, \textbf{42} (2000), 41--61
	\newblock DOI: 10.1016/S0362-546X(98)00329-0
	
	
	
	
\end{thebibliography}
	\end{document}